\RequirePackage{fix-cm}
\documentclass[oneside,english,british]{amsart}
\usepackage{lmodern}
\usepackage[T1]{fontenc}
\usepackage[latin9]{inputenc}
\usepackage{geometry}
\usepackage{color}
\usepackage{babel}
\usepackage{prettyref}
\usepackage{amsthm}
\usepackage{amstext}
\usepackage{amssymb}
\usepackage{fixltx2e}
\usepackage{esint}
\usepackage[unicode=true,pdfusetitle,
 bookmarks=true,bookmarksnumbered=false,bookmarksopen=false,
 breaklinks=false,pdfborder={0 0 0},backref=false,colorlinks=false]
 {hyperref}

\makeatletter
\numberwithin{equation}{section}
\numberwithin{figure}{section}
\theoremstyle{plain}
\newtheorem{thm}{\protect\theoremname}[section]
  \theoremstyle{remark}
  \newtheorem{rem}[thm]{\protect\remarkname}
 \newcommand\thmsname{\protect\theoremname}
 \newcommand\nm@thmtype{theorem}
 \theoremstyle{plain}
 
 \newenvironment{namedthm}[1][Undefined Theorem Name]{
   \ifx{#1}{Undefined Theorem Name}\renewcommand\nm@thmtype{theorem*}
   \else\renewcommand\thmsname{#1}\renewcommand\nm@thmtype{namedtheorem}
   \fi
   \begin{\nm@thmtype}}
   {\end{\nm@thmtype}}
  \theoremstyle{plain}
  \newtheorem{prop}[thm]{\protect\propositionname}
  \theoremstyle{plain}
  \newtheorem{lem}[thm]{\protect\lemmaname}
  \theoremstyle{remark}
  \newtheorem{claim}[thm]{\protect\claimname}
  \theoremstyle{plain}
  \newtheorem{cor}[thm]{\protect\corollaryname}

\usepackage{bbold}

\makeatother

  \addto\captionsbritish{\renewcommand{\claimname}{Claim}}
  \addto\captionsbritish{\renewcommand{\corollaryname}{Corollary}}
  \addto\captionsbritish{\renewcommand{\lemmaname}{Lemma}}
  \addto\captionsbritish{\renewcommand{\propositionname}{Proposition}}
  \addto\captionsbritish{\renewcommand{\remarkname}{Remark}}
  \addto\captionsbritish{\renewcommand{\theoremname}{Theorem}}
  \addto\captionsenglish{\renewcommand{\claimname}{Claim}}
  \addto\captionsenglish{\renewcommand{\corollaryname}{Corollary}}
  \addto\captionsenglish{\renewcommand{\lemmaname}{Lemma}}
  \addto\captionsenglish{\renewcommand{\propositionname}{Proposition}}
  \addto\captionsenglish{\renewcommand{\remarkname}{Remark}}
  \addto\captionsenglish{\renewcommand{\theoremname}{Theorem}}
  \providecommand{\claimname}{Claim}
  \providecommand{\corollaryname}{Corollary}
  \providecommand{\lemmaname}{Lemma}
  \providecommand{\propositionname}{Proposition}
  \providecommand{\remarkname}{Remark}
  \providecommand{\theoremname}{Theorem}
\providecommand{\theoremname}{Theorem}

\begin{document}

\title{Equidistribution of values of linear forms on quadratic surfaces.}

\author{Oliver Sargent}

\address{\textsc{Department Of Mathematics, University Walk, Bristol, BS8
1TW, UK.}\texttt{ }}

\email{\texttt{Oliver.Sargent@bris.ac.uk}}
\begin{abstract}
In this paper we investigate the distribution of the set of values
of a linear map at integer points on a quadratic surface. In particular,
it is shown that subject to certain algebraic conditions, this set
is equidistributed. This can be thought of as a quantitative version
of the main result from \cite{2011arXiv1111.4428S}. The methods used
are based on those developed by A. Eskin, S. Mozes and G. Margulis
in \cite{MR1609447}. Specifically, they rely on equidistribution
properties of unipotent flows. 
\end{abstract}
\maketitle

\section{Introduction.}

Consider the following situation. Let $X$ be a rational surface in
$\mathbb{R}^{d}$, $R$ be a fixed region in \foreignlanguage{english}{$\mathbb{R}^{s}$}
and $F:X\rightarrow\mathbb{R}^{s}$ be a polynomial map. An interesting
problem is to investigate the size of the set 
\[
Z=\left\{ x\in X\cap\mathbb{Z}^{d}:F\left(x\right)\in R\right\} ,
\]
consisting of integer points in $X$ such that the corresponding values
of $F$, are in $R$. Suppose that the set of values of $F$ at the
integer points of $X$, is dense in $\mathbb{R}^{s}$. In this case,
the set $Z$ will be infinite. However, the set 
\[
Z_{T}=\left\{ x\in X\cap\mathbb{Z}^{d}:F\left(x\right)\in R,\left\Vert x\right\Vert \leq T\right\} ,
\]
can be considered. This set will be finite, and its size will depend
on $T$. Typically, the density assumption indicates that the set
$Z$ might be equidistributed, within the set of all integer points
in $X$. Namely, as $T$ increases, the size of the set $Z_{T}$,
should be proportional to the appropriately defined volume, of the
set 
\[
\left\{ x\in X:F\left(x\right)\in R,\left\Vert x\right\Vert \leq T\right\} ,
\]
consisting of real points on $X$, with values in $R$ and bounded
norm. Such a result, if it is obtained, can be seen as quantifying
the denseness of the values of $F$ at integral points.

The situation described above is too general, but it serves as motivation
for what is to come. So far, what is proven, is limited to special
cases. For instance, when $M:\mathbb{R}^{d}\rightarrow\mathbb{R}^{s}$
is a linear map, classical methods can be used to establish necessary
and sufficient conditions, which ensure the values of $M$ on $\mathbb{Z}^{d}$
are dense in $\mathbb{R}^{s}$. The equidistribution problem described
above can also be considered in this case. It is straightforward to
obtain an asymptotic estimate for the number of integer points with
bounded norm whose values lie in some compact region of $\mathbb{R}^{s}$
(cf. \cite{MR0349591}).

When $Q:\mathbb{R}^{d}\rightarrow\mathbb{R}$ is a quadratic form
the situation is that of the Oppenheim conjecture. In \cite{MR993328},
G. Margulis obtained necessary and sufficient conditions to ensure
that the values of $Q$ on $\mathbb{Z}^{d}$ are dense in $\mathbb{R}$.
Considerable work has gone into the equidistribution problem in this
case, first by S.G. Dani and G. Margulis, who obtained an asymptotic
lower bound for the number of integers with bounded height such that
their images lie in a fixed interval (cf. \cite{MR1237827}). Later,
A. Eskin, G. Margulis and S. Mozes, gave the corresponding asymptotic
upper bound for the same problem (cf. \cite{MR1609447}). The major
ingredient, used in the proof of Oppenheim conjecture, is to relate
the density of the values of a quadratic form at integers to the density
of certain orbits inside a homogeneous space. This connection was
first noted by M. S. Raghunathan in the late 70's (appearing in print
in \cite{MR629475}, for instance). It is, in this way, using tools
from dynamical systems to study the orbit closures of subgroups corresponding
to quadratic forms, that Margulis proved the Oppenheim conjecture.
Similarly, the later refinement, due to Dani-Margulis, who considered
the values of quadratic forms at primitive integral points in \cite{MR1032925}
and work on the equidistribution (quantitative) problem by Dani-Margulis
and Eskin-Margulis-Mozes, were also obtained by studying the orbit
closures of subgroups acting on homogeneous spaces. 

Similar techniques were also used by A. Gorodnik in \cite{MR2067128},
to study the set of values of a pair, consisting of a quadratic and
linear form, at integer points and in \cite{2011arXiv1111.4428S}
to establish conditions, sufficient to ensure that the values of a
linear map at integers lying on a quadratic surface are dense in the
range of the map. The main result of this paper deals with the corresponding
equidistribution problem and is stated in the following Theorem. 
\begin{thm}
\label{main theorem} Suppose $Q$ is a quadratic form on $\mathbb{R}^{d}$
such that $Q$ is non-degenerate, indefinite with rational coefficients.
Let $M=\left(L_{1},\ldots,L_{s}\right):\mathbb{R}^{d}\rightarrow\mathbb{R}^{s}$
be a linear map such that: 
\begin{enumerate}
\item \label{enu:condition 1}The following relations hold, $d>2s$ and
$\textrm{rank}\left(Q|_{\ker\left(M\right)}\right)=d-s$. 
\item \label{enu:condition 2}The quadratic form $Q|_{\ker\left(M\right)}$
has signature $\left(r_{1},r_{2}\right)$ where $r_{1}\geq3$ and
$r_{2}\geq1$.
\item \label{enu:condition 3}For all $\alpha\in\mathbb{R}^{s}\setminus\left\{ 0\right\} $,
$\alpha_{1}L_{1}+\dots+\alpha_{s}L_{s}$ is non rational. 
\end{enumerate}
Let $a\in\mathbb{Q}$ be such that the set $\left\{ v\in\mathbb{Z}^{d}:Q\left(v\right)=a\right\} $
is non empty. Then there exists $C_{0}>0$ such that for every $\theta>0$
and all compact $R\subset\mathbb{R}^{s}$ with piecewise smooth boundary,
there exists a $T_{0}>0$ such that for all $T>T_{0}$, 
\[
\left(1-\theta\right)C_{0}\textrm{Vol}\left(R\right)T^{d-s-2}\leq\left|\left\{ v\in\mathbb{Z}^{d}:Q\left(v\right)=a,M\left(v\right)\in R,\left\Vert v\right\Vert \leq T\right\} \right|\leq\left(1+\theta\right)C_{0}\textrm{Vol}\left(R\right)T^{d-s-2},
\]
where $\textrm{Vol}\left(R\right)$ is the $s$ dimensional Lebesgue
measure of $R$. \end{thm}
\begin{rem}
The constant $C_{0}$ appearing in Theorem \ref{main theorem} is
such that 
\[
C_{0}\textrm{Vol}\left(R\right)T^{d-s-2}\sim\textrm{Vol}\left(\left\{ v\in\mathbb{\mathbb{R}}^{d}:Q\left(v\right)=a,M\left(v\right)\in R,\left\Vert v\right\Vert \leq T\right\} \right)
\]
where the volume on the right is the Haar measure on the surface defined
by $Q\left(v\right)=a$. 
\end{rem}

\begin{rem}
Theorem \ref{main theorem} should hold with the condition that $\textrm{rank}\left(Q|_{\ker\left(M\right)}\right)=d-s$
replaced by the condition that $\textrm{rank}\left(Q|_{\ker\left(M\right)}\right)>3$.
Dealing with the more general situation requires taking into account
the nontrivial unipotent part of $\textrm{Stab}_{SO\left(Q\right)}\left(M\right)$,
as such lower bounds could probably be proved using methods of $[DM93]$,
but so far no way has been found to obtain the statement that would
be needed in order to obtain an upper bound. 
\end{rem}

\begin{rem}
As in \cite{MR1609447} it would be possible to obtain a version of
Theorem \ref{main theorem} where the condition that $\left\Vert v\right\Vert <T$
was replaced by $v\in TK_{0}$ where $K_{0}$ is an arbitrary deformation
of the unit ball by a continuous and positive function. It should
also be possible to obtain a version of Theorem \ref{main theorem}
where the parameters $T_{0}$ and $C_{0}$ remain valid for any pair
$\left(Q,M\right)$ coming from compact subsets of pairs satisfying
the conditions of the Theorem. 
\end{rem}

\begin{rem}
The cases when the quadratic form $Q|_{\ker\left(M\right)}$ has signature
$\left(2,2\right)$ or $\left(2,1\right)$ can be considered exceptional.
There are asymptotically more integers than expected (by a factor
of $\log T$) lying on certain surfaces defined by quadratic forms
of signature $\left(2,2\right)$ or $\left(2,1\right)$. This leads
to counterexamples of Theorem \ref{main theorem} in the cases when
the quadratic form $Q|_{\ker\left(M\right)}$ has signature $\left(2,2\right)$
or $\left(2,1\right)$. Details of these examples are found in \prettyref{sec:Counterexamples}.
\end{rem}

\subsection*{Outline of the paper}

The proof of Theorem \ref{main theorem} rests on statements about
the distribution of orbits in a certain homogeneous spaces. The philosophy
is that equidistribution of the orbits corresponds to equidistribution
of the points considered in Theorem \ref{main theorem}. Consider
the following Theorem of M. Ratner found in \cite{MR1262705}.
\begin{namedthm}[Ratner's Equidistribution Theorem]
Let $G$ be a connected Lie group, $\Gamma$ a lattice in $G$ and
$U=\left\{ u_{t}:t\in\mathbb{R}\right\} $ a one parameter unipotent
subgroup of $G$. Then for all $x\in G/\Gamma$ the closure of the
orbit $Ux$ has an invariant measure, $\mu_{\overline{Ux}}$ supported
on it and for all bounded continuous functions, $f$ on $G/\Gamma$,
\[
\lim_{T\rightarrow\infty}\frac{1}{T}\int_{0}^{T}f\left(u_{t}x\right)=\int_{\overline{Ux}}fd\mu_{\overline{Ux}}.
\]

\end{namedthm}
Recall that in the proof of the quantitative Oppenheim conjecture
(cf. \cite{MR1609447}) one needs to consider an unbounded function
on the space of lattices. Similarly, in order to prove Theorem \ref{main theorem}
one needs to consider an unbounded function $F$ on a certain homogeneous
space. The basic idea is to try to apply Ratner's Equidistribution
Theorem to $F$ in order to show that the average of the values of
$F$ evaluated along a certain orbit converges to the average of $F$
on the entire space. This is the fact that corresponds to the fact
that integral points on the quadratic surface with values in $R$
are equidistributed. The main problem in doing this is that $F$ is
unbounded and so one must obtain an ergodic theorem, taking a similar
form to Ratner's Equidistribution Theorem, but valid for unbounded
functions. In order to do this one needs precise information about
the behaviour of the orbits near the cusp. This information is obtained
in \prettyref{sec:The-alpha-function} and comes in the form of non
divergence estimates for certain dilated spherical averages. In order
to obtain these estimates we use a certain function defined by Y.
Benoist and J.F. Quint in \cite{MR2874934}. The required ergodic
theorem is then proved in \prettyref{sec:Ergodic-theorems.}. Finally
in \prettyref{sec:Proof-of-Theorem} the proof of Theorem \ref{main theorem}
is completed using an approximation argument similar to that found
in \cite{MR1609447}. Specifically, the averages of $F$ over the
space are related to the quantity $C_{0}\textrm{Vol}\left(R\right)T^{d-s-2}$
and the averages of $F$ along an orbit are related to the number
of integer points with bounded height, lying on the surface and with
values in $R$. In \prettyref{sec:Set-up.} the basic notation is
set up and the main results from \prettyref{sec:The-alpha-function}
and \prettyref{sec:Ergodic-theorems.} are stated.

\subsection*{Acknowledgements}

The author would like to thank Alex Gorodnik for many helpful discussions
and remarks about earlier versions of this paper.

\section{Set up.\label{sec:Set-up.}}

\subsection{Main results. }

For the rest of the paper the following convention is in place: $s,d$
and $p$ will be fixed natural numbers such that $2s<d$ and $0<p<d$.
Also, $r_{1}$ and $r_{2}$ will be varying, natural numbers such
that $d-s=r_{1}+r_{2}$. Let $\mathcal{L}$ denote the space of linear
forms on $\mathbb{R}^{d}$ and let $\mathcal{C}_{\textrm{Lin}}$ denote
the subset of $\mathcal{L}^{s}$ such that for all $M\in\mathcal{C}_{\textrm{Lin}}$
condition \ref{enu:condition 3} of Theorem \ref{main theorem} is
satisfied. A quadratic form on $\mathbb{R}^{d}$ is said to be defined
over $\mathbb{Q}$, if it has rational coefficients or is a scalar
multiple of a form with rational coefficients. For $a$, a rational
number let $\mathcal{Q}\left(p,a\right)$ denote quadratic forms on
$\mathbb{R}^{d}$ defined over $\mathbb{Q}$ with signature $\left(p,d-p\right)$
such that the set $\left\{ v\in\mathbb{Z}^{d}:Q\left(v\right)=a\right\} $
is non empty for all $Q\in\mathcal{Q}\left(p,a\right)$. Define
\[
\mathcal{\mathcal{C}_{\textrm{Pairs}}}\left(a,r_{1},r_{2}\right)=\left\{ \left(Q,M\right):Q\in\mathcal{Q}\left(p,a\right),M\in\mathcal{\mathcal{C}_{\textrm{Lin}}}\textrm{ and }Q|_{\ker\left(M\right)}\textrm{ has signature }\left(r_{1},r_{2}\right)\right\} .
\]
Note that for $r_{1}\geq3$ and $r_{2}\geq1$ the set $\mathcal{\mathcal{C}_{\textrm{Pairs}}}\left(a,r_{1},r_{2}\right)$
consists of pairs satisfying the conditions of Theorem \ref{main theorem}.
Although the set $\mathcal{\mathcal{C}_{\textrm{Pairs}}}\left(a,r_{1},r_{2}\right)$
and hence its subsets and sets derived from them depend on $a$, this
dependence is not a crucial one, so from now on, most of the time
this dependence will be omitted from the notation. For $M\in\mathcal{L}^{s}$
and $R\subset\mathbb{R}^{s}$ a connected region with smooth boundary
let $V_{M}\left(R\right)=\left\{ v\in\mathbb{R}^{d}:M\left(v\right)\in R\right\} .$
For $Q\in\mathcal{Q}\left(p,d-p\right)$, $a\in\mathbb{Q}$ and $\mathbb{K}=\mathbb{R}$
or $\mathbb{Z}$ let $X_{Q}^{a}\left(\mathbb{K}\right)=\left\{ v\in\mathbb{K}^{d}:Q\left(v\right)=a\right\} .$
Denote the annular region inside $\mathbb{R}^{d}$ by $A\left(T_{1},T_{2}\right)=\left\{ v\in\mathbb{R}^{d}:T_{1}\leq\left\Vert v\right\Vert \leq T_{2}\right\} .$
Using this notation, we state the following (equivalent) version of
Theorem \ref{main theorem}, which will be proved in \prettyref{sec:Proof-of-Theorem}.
\begin{thm}
\label{main theorem-1}Suppose that $r_{1}\geq3$, $r_{2}\geq1$ and
$a\in\mathbb{Q}$. Then for all $\left(Q,M\right)\in\mathcal{\mathcal{C}_{\textrm{Pairs}}}\left(a,r_{1},r_{2}\right)$
there exists $C_{0}>0$ such that for every $\theta>0$ and all compact
$R\subset\mathbb{R}^{s}$ with piecewise smooth boundary, there exists
a $T_{0}>0$ such that for all $T>T_{0}$, 
\[
\left(1-\theta\right)C_{0}\textrm{Vol}\left(R\right)T^{d-s-2}\leq\left|X_{Q}^{a}\left(\mathbb{Z}\right)\cap V_{M}\left(R\right)\cap A\left(0,T\right)\right|\leq\left(1-\theta\right)C_{0}\textrm{Vol}\left(R\right)T^{d-s-2}.
\]
\end{thm}
\begin{rem}
As remarked previously, the cases when $r_{1}=2$ and $r_{2}=2$ or
$r_{1}=2$ and $r_{2}=1$ are interesting. In dimensions 3 and 4 there
can be more integer points than expected lying on some surfaces defined
by quadratic forms of signature $\left(2,2\right)$ or $\left(2,1\right)$,
this means that the statement of Theorem \ref{main theorem-1} fails
for certain pairs. In Section \ref{sec:Counterexamples} these counterexamples
are explicitly constructed. Moreover, it is shown that this set of
pairs is big in the sense that it is of second category. We note that
as in \cite{MR1609447} one could also show that this set has measure
zero and one could prove the expected asymptotic formula as in Theorem
\ref{main theorem-1} for almost all pairs.
\end{rem}
Even though Theorem \ref{main theorem-1} fails when $r_{1}=2$ and
$r_{2}=2$ or $r_{1}=2$ and $r_{2}=1$, we do have the following
uniform upper bound, which will be proved in Section \ref{sec:Proof-of-Theorem}
and is analogous to Theorem 2.3 from \cite{MR1609447}. 
\begin{thm}
\label{prop:crude upper bound} Let $R\subset\mathbb{R}^{s}$ be a
compact region with piecewise smooth boundary and $a\in\mathbb{Q}$.

\selectlanguage{english}%
    \renewcommand{\labelenumi}{(\Roman{enumi})}
\selectlanguage{british}%
\begin{enumerate}
\item If $r_{1}\geq3$ and $r_{2}\geq1$, then for all $\left(Q,M\right)\in\mathcal{\mathcal{C}_{\textrm{Pairs}}}\left(a,r_{1},r_{2}\right)$,
there exists a constant $C$ depending only on $\left(Q,M\right)$
and $R$ such that for all $T>1$,
\[
\left|X_{Q}^{a}\left(\mathbb{Z}\right)\cap V_{M}\left(R\right)\cap A\left(0,T\right)\right|\leq CT^{d-s-2}.
\]

\item If $r_{1}=2$ and $r_{2}=1$ or $r_{1}=r_{2}=2$ , then for all $\left(Q,M\right)\in\mathcal{\mathcal{C}_{\textrm{Pairs}}}\left(a,r_{1},r_{2}\right)$
there exists a constant $C$ depending only on $\left(Q,M\right)$
and $R$ such that for all $T>2$,
\[
\left|X_{Q}^{a}\left(\mathbb{Z}\right)\cap V_{M}\left(R\right)\cap A\left(0,T\right)\right|\leq C\left(\log T\right)T^{d-s-2}.
\]

\end{enumerate}
\end{thm}

\subsection{A canonical form. }

For $v_{1},v_{2}\in\mathbb{R}^{d}$ we will use the notation $\left\langle v_{1},v_{2}\right\rangle $
to denote the standard inner product in $\mathbb{R}^{d}$. For a set
of vectors $v_{1},\dots,v_{i}\in\mathbb{R}^{d}$ we will also use
the notation $\left\langle v_{1},\dots,v_{i}\right\rangle $ to denote
the span of $v_{1},\dots,v_{i}$ in $\mathbb{R}^{d}$, although this
could lead to some ambiguity, the meaning of the notation should be
clear from the context. 

For some computations it will be convenient to know that our system
is conjugate to a canonical form. Let $e_{1},\dots,e_{d}$ be the
standard basis of $\mathbb{R}^{d}$. Let $\left(Q_{0},M_{0}\right)$
be the pair consisting of a quadratic form and a linear map defined
by 
\[
Q_{0}\left(v\right)=Q_{1,\ldots,s}\left(v\right)+2v_{s+1}v_{d}+\sum_{i=s+2}^{s+r_{1}}v_{i}^{2}-\sum_{i=s+r_{1}+1}^{d-1}v_{i}^{2}\quad\textrm{ and }\quad M_{0}\left(v\right)=\left(v_{1},\ldots,v_{s}\right),
\]
where $v_{i}=\left\langle v,e_{i}\right\rangle $ and $Q_{1,\ldots,s}\left(v\right)$
is a non degenerate quadratic form in variables $v_{1},\ldots,v_{s}$.
By Lemma 2.2 of \cite{2011arXiv1111.4428S} all pairs $\left(Q,M\right)$
such that $\textrm{rank}\left(Q|_{\ker\left(M\right)}\right)=d-s$
and the signature of $Q|_{\ker\left(M\right)}$ is $\left(r_{1},r_{2}\right)$
are equivalent to the pair $\left(Q_{0},M_{0}\right)$ in the sense
that there exist $g_{d}\in GL_{d}\left(\mathbb{R}\right)$ and $g_{s}\in GL_{s}\left(\mathbb{R}\right)$
such that $\left(Q,M\right)=\left(Q_{0}^{g_{d}},g_{s}M_{0}^{g_{d}}\right)$,
where for $g\in GL_{d}\left(\mathbb{R}\right)$ we write $Q=Q_{0}^{g}$
if and only if $Q_{0}\left(gv\right)=Q\left(v\right)$ for all $v\in\mathbb{R}^{d}$.
Moreover since $R\subset\mathbb{R}^{s}$ is arbitrary, up to rescaling
and possibly replacing $R$ by $g_{s}R$ we assume that $g_{d}\in SL_{d}\left(\mathbb{R}\right)$
and that $g_{s}$ is the identity. Let 
\[
\mathcal{C}_{\textrm{SL}}\left(a,r_{1},r_{2}\right)=\left\{ g\in SL_{d}\left(\mathbb{R}\right):\left(Q_{0}^{g},M_{0}^{g}\right)\in\mathcal{\mathcal{C}_{\textrm{Pairs}}}\left(a,r_{1},r_{2}\right)\right\} .
\]
For $g\in\mathcal{C}_{\textrm{SL}}\left(a,r_{1},r_{2}\right)$, let
$G_{g}$ be the identity component of the group $\left\{ x\in SL_{d}\left(\mathbb{R}\right):Q_{0}^{g}\left(xv\right)=Q_{0}^{g}\left(v\right)\right\} $,
$\Gamma_{g}=G_{g}\cap SL_{d}\left(\mathbb{Z}\right)$, $H_{g}=\left\{ x\in G_{g}:M_{0}^{g}\left(xv\right)=M_{0}^{g}\left(v\right)\right\} $
and $K_{g}=H_{g}\cap g^{-1}O_{d}\left(\mathbb{R}\right)g$. By examining
the description of the subgroup $H_{g}$, given in Section 2.3 of
\cite{2011arXiv1111.4428S} it is clear that $K_{g}$ is a maximal
compact subgroup of $H_{g}$. It is a standard fact that $G_{g}$
is a connected semisimple Lie group and hence, has no nontrivial rational
characters. Therefore, because $Q_{0}^{g}$ is defined over $\mathbb{Q}$,
the Borel Harish-Chandra Theorem (cf. \cite{MR1278263}, Theorem 4.13)
implies $\Gamma_{g}$ is a lattice in $G_{g}$. We will consider the
dynamical system that arises from $H_{g}$ acting on $G_{g}/\Gamma_{g}$.
For $\mathbb{K}=\mathbb{R}$ or $\mathbb{Z}$, the shorthand $X_{Q_{0}^{g}}^{a}\left(\mathbb{K}\right)=X_{g}\left(\mathbb{K}\right)$
will be used.

\subsection{Equidistribution of measures\label{sub:The alpha func}}

Consider the function $\alpha$, as defined in \cite{MR1609447}.
It is an unbounded function on the space of unimodular lattices in
$\mathbb{R}^{d}$. It has the properties that it can be used to bound
certain functions that we will consider and it is left $K_{I}$ invariant.
Similar functions have been considered in \cite{MR1314965} where
it is related to various quantities involving successive minima of
a lattice. Let $\Delta$ be a lattice in $\mathbb{R}^{d}$. For any
such $\Delta$ we say that a subspace $U$ of $\mathbb{R}^{d}$ is
$\Delta$-rational if $\textrm{Vol}\left(U/U\cap\Delta\right)<\infty$.
Let 
\[
\Psi_{i}\left(\Delta\right)=\left\{ U:U\textrm{ is a }\Delta\textrm{-rational subspace of }\mathbb{R}^{d}\textrm{ with }\dim U=i\right\} .
\]
For $U\in\Psi_{i}\left(\Delta\right)$ define $d_{\Delta}\left(U\right)=\textrm{Vol}\left(U/U\cap\Delta\right)$.
Note that $d_{\Delta}\left(U\right)=\left\Vert u_{1}\wedge\ldots\wedge u_{i}\right\Vert $
where $u_{1},\ldots,u_{i}$ is a basis for $U\cap\Delta$ over $\mathbb{Z}$
and the norm on $\bigwedge^{i}\left(\mathbb{R}^{d}\right)$ is induced
from the euclidean norm on $\mathbb{R}^{d}$. Now we recall the definition
the function $\alpha$, as follows 
\[
\alpha_{i}\left(\Delta\right)=\sup_{U\in\Psi_{i}\left(\Delta\right)}\frac{1}{d_{\Delta}\left(U\right)}\quad\textrm{and}\quad\alpha\left(\Delta\right)=\max_{0\leq i\leq d}\alpha_{i}\left(\Delta\right).
\]
Here we use the convention that if $U$ is the trivial subspace then
$d_{\Delta}\left(U\right)=1$, hence $\alpha_{0}\left(\Delta\right)=1$.
Also note that if $\Delta$ is a unimodular lattice then $d_{\Delta}\left(\mathbb{R}^{d}\right)=1$
and hence $\alpha_{d}\left(\Delta\right)=1$. 

In \eqref{eq:upper bound for function} and Theorem \ref{thm:ergodic result-1}
we consider $\alpha$ as a function on $G_{g}/\Gamma_{g}$, this is
done via the canonical embedding of $G_{g}/\Gamma_{g}$ into the space
of unimodular lattices in $\mathbb{R}^{d}$, given by $x\Gamma_{g}\rightarrow x\mathbb{Z}^{d}$.
Specifically, every $x\in G_{g}/\Gamma_{g}$ can be identified with
its image under this embedding before applying $\alpha$ to it. For
$f\in C_{c}\left(\mathbb{R}^{d}\right)$ and $g\in\mathcal{C}_{\textrm{SL}}\left(r_{1},r_{2}\right)$
we define the function $F_{f,g}:G_{g}/\Gamma_{g}\rightarrow\mathbb{R}$
by 
\begin{equation}
F_{f,g}\left(x\right)=\sum_{v\in X_{g}\left(\mathbb{Z}\right)}f\left(xv\right).\label{eq:def of F}
\end{equation}
The function $\alpha$ has the property that there exists a constant
$c\left(f\right)$ depending only on the support and maximum of $f$
such that for all $x$ in $G_{g}/\Gamma_{g}$,
\begin{equation}
F_{f,g}\left(x\right)\leq c\left(f\right)\alpha\left(x\right).\label{eq:upper bound for function}
\end{equation}
The last property is well known and follows from Minkowski's Theorem
on successive minima, see Lemma 2 of \cite{MR0224562} for example.
Alternatively, see \cite{MR2454302} for an up to date review of many
related results. 

We will be carrying out integration on various measure spaces defined
by the groups introduced at the beginning of the section. With this
in mind let us introduce the following notation for the corresponding
measures. If $v$ denotes some variable, the notation $dv$ is used
to denote integration with respect to Lebesgue measure and this variable.
Let $\mu_{g}$ be the Haar measure on $G_{g}/\Gamma_{g}$, if $g\in\mathcal{C}_{\textrm{SL}}\left(r_{1},r_{2}\right)$
then since $\Gamma_{g}$ is a lattice in $G_{g}$ we can normalise
so that $\mu_{g}\left(G_{g}/\Gamma_{g}\right)=1$. In addition, $\nu_{g}$
will denote the measure on $K_{g}$ normalised so that $\nu_{g}\left(K_{g}\right)=1$.
Let $m_{g}^{a}$ denote the Haar measure on $X_{g}^{a}\left(\mathbb{R}\right)$
defined by 
\begin{equation}
\int_{\mathbb{R}^{d}}f\left(v\right)dv=\int_{-\infty}^{\infty}\int_{X_{g}^{a}\left(\mathbb{R}\right)}f\left(v\right)dm_{g}^{a}\left(v\right)da.\label{eq:def of m_0}
\end{equation}
The following Theorem provides us with our upper bounds and will be
proved in Section \ref{sec:The-alpha-function}. 
\begin{thm}
\label{thm:upper bound-1}Let $g\in\mathcal{C}_{\textrm{SL}}\left(r_{1},r_{2}\right)$
be arbitrary and let $\Delta=g\mathbb{Z}^{d}$. Let $\left\{ a_{t}:t\in\mathbb{R}\right\} $
denote a self adjoint one parameter subgroup of $SO\left(2,1\right)$
embedded into $H_{I}$ so that it fixes the subspace $\left\langle e_{s+2},\ldots,e_{d-1}\right\rangle $
and only has eigenvalues $e^{-t},$ $1$ and $e^{t}$. 

\selectlanguage{english}%
    \renewcommand{\labelenumi}{(\Roman{enumi})}
\selectlanguage{british}%
\begin{enumerate}
\item Suppose $r_{1}\geq3$ , $r_{2}\geq1$ and $0<\delta<2$, then
\[
\sup_{t>0}\int_{K_{I}}\alpha\left(a_{t}k\Delta\right)^{\delta}d\nu_{I}\left(k\right)<\infty.
\]

\item Suppose $r_{1}=r_{2}=2$ or $r_{1}=2$, $r_{2}=1$, then
\[
\sup_{t>1}\frac{1}{t}\int_{K_{I}}\alpha\left(a_{t}k\Delta\right)d\nu_{I}\left(k\right)<\infty.
\]

\end{enumerate}
\end{thm}
In Section \ref{sec:Ergodic-theorems.} we will modify the results
from Section 4 of \cite{MR1609447} and combine them with Theorem
\ref{thm:upper bound-1} to prove the following Theorem which will
be a major ingredient of the proof of Theorem \ref{main theorem-1}. 
\begin{thm}
\label{thm:ergodic result-1}Suppose $r_{1}\geq3$ and $r_{2}\geq1$.
Let $A=\left\{ a_{t}:t\in\mathbb{R}\right\} $ be a one parameter
subgroup of $H_{g}$, not contained in any proper normal subgroup
of $H_{g}$, such that there exists a continuous homomorphism $\rho:SL_{2}\left(\mathbb{R}\right)\rightarrow H_{g}$
with $\rho\left(D\right)=A$ and $\rho\left(SO\left(2\right)\right)\subset K_{g}$
where $D=\left\{ \left(\begin{smallmatrix}t & 0\\
0 & t^{-1}
\end{smallmatrix}\right):t>0\right\} $. Let $\phi\in L^{1}\left(G_{g}/\Gamma_{g}\right)$ be a continuous
function such that for some $0<\delta<2$ and some $C>0$, 
\begin{equation}
\left|\phi\left(\Delta\right)\right|<C\alpha\left(\triangle\right)^{\delta},\textrm{ for all }\Delta\in G_{g}/\Gamma_{g}.\label{eq:bound in thm 2.4}
\end{equation}
Then for all $\epsilon>0$ and all $g\in\mathcal{C}_{\textrm{SL}}\left(r_{1},r_{2}\right)$
there exists $T_{0}>0$ such that for all $t>T_{0}$, 
\[
\left|\int_{K_{g}}\phi\left(a_{t}k\right)d\nu_{g}\left(k\right)-\int_{G_{g}/\Gamma_{g}}\phi d\mu_{g}\right|\leq\epsilon.
\]

\end{thm}

\begin{rem}
The condition that $A$ should not be contained in any proper normal
subgroup of $H_{g}$ is only necessary in the case when $H_{g}\cong SO\left(2,2\right)$,
since in all other cases $H_{g}$ is simple. 
\end{rem}

\section{The upper bounds.\label{sec:The-alpha-function}}

In this section we prove Theorem \ref{thm:upper bound-1}. By definition
$H_{I}\cong SO\left(r_{1},r_{2}\right)$ and is embedded in $SL_{d}\left(\mathbb{R}\right)$
so that it fixes $\left\langle e_{1},\ldots,e_{s}\right\rangle $.
Let $\left\{ a_{t}:t\in\mathbb{R}\right\} $ denote a self adjoint
one parameter subgroup of $SO\left(2,1\right)$ embedded into $H_{I}$
so that it fixes the subspace $\left\langle e_{s+2},\ldots,e_{d-1}\right\rangle $.
Moreover, suppose that the only eigenvalues of $a_{t}$ are $e^{-t},1$
and $e^{t}$. For $g\in\mathcal{C}_{\textrm{SL}}\left(r_{1},r_{2}\right)$,
let $\Delta=g\mathbb{Z}^{d}$.

\subsection{Proof of part I of Theorem \ref{thm:upper bound-1}.}

The aim is to construct a function $f:H_{I}\to\mathbb{R}$ which is
contracted by the operator 
\[
A_{t}f\left(h\right)=\int_{K_{I}}f\left(a_{t}kh\right)d\nu_{I}\left(k\right).
\]
We say that $f$ is contracted by the operator $A_{t}$ if for any
$c>0$ there exists $t_{0}>0$ and $b>0$ such that for all $h\in H_{I}$,
\[
A_{t_{0}}f\left(h\right)<cf\left(h\right)+b.
\]
This fact will be used in conjunction with Proposition 5.12 from \cite{MR1609447}
which is stated below.
\begin{prop}
\label{prop:Let--be}Let $f:H_{I}\to\mathbb{R}$ be a strictly positive
function such that:
\begin{enumerate}
\item \label{enu:contraction 1}For any $\epsilon>0$ there exists a neighbourhood
$V\left(\epsilon\right)$ of $1$ in $H_{I}$ such that 
\[
\left(1-\epsilon\right)f\left(h\right)\leq f\left(uh\right)\leq\left(1+\epsilon\right)f\left(h\right)
\]
for all $h\in H_{I}$ and $u\in V\left(\epsilon\right)$.
\item \label{enu:contraction 2}The function $f$ is left $K_{I}$ invariant. 
\item \label{enu:contraction 3}$f\left(1\right)<\infty$. 
\item \label{enu:contraction 4}The function $f$ is contracted by the operator
$A_{t}$.
\end{enumerate}
Then \textup{$\sup_{t>0}A_{t}f\left(1\right)<\infty$}.
\end{prop}
It is clear that if in addition to satisfying properties \eqref{enu:contraction 1}-\eqref{enu:B+Q, convexity 4},
we have $\alpha\left(h\Delta\right)^{\delta}\leq f\left(h\right)$
for all $h\in H_{I}$, then the conclusion of Part I of Theorem \ref{thm:upper bound-1}
follows. We define the function in three stages. In the first stage
we define a function on the exterior algebra of $\mathbb{R}^{d}$,
then this function is used to define a function on the space of lattices
in $\mathbb{R}^{d}$. Finally we use that function to define a function
with the required properties.

\subsubsection{A function on the exterior algebra of $\mathbb{R}^{d}$. }

Let $\bigwedge\left(\mathbb{R}^{d}\right)=\bigoplus_{i=1}^{d-1}\bigwedge^{i}\left(\mathbb{R}^{d}\right)$.
We say that $v\in\bigwedge\left(\mathbb{R}^{d}\right)$ has degree
$i$ if $v\in\bigwedge^{i}\left(\mathbb{R}^{d}\right)$. Let $\Omega_{i}=\left\{ v_{1}\wedge\dots\wedge v_{i}:v_{1},\dots,v_{i}\in\mathbb{R}^{d}\right\} $
be the set of monomial elements of $\bigwedge\left(\mathbb{R}^{d}\right)$
with degree $i$. Define $\Omega=\bigcup_{i=1}^{d-1}\Omega_{i}$.
Consider the representation $\rho:H_{I}\to GL\left(\bigwedge\left(\mathbb{R}^{d}\right)\right)$.
Since $H_{I}$ is semisimple this representation decomposes as a direct
sum of irreducible subrepresentations. Associated to each of these
subrepresentations is a unique highest weight. Let $\mathcal{P}$
denote the set of all these highest weights. For $\lambda\in\mathcal{P}$,
denote by $U^{\lambda}$ the sum of all of the subrepresentations
with highest weight $\lambda$ and let $\tau_{\lambda}:\bigwedge\left(\mathbb{R}^{d}\right)\to U^{\lambda}$
be the orthogonal projection. 

Let $\epsilon>0$. For $0<i<d$ and $v\in\bigwedge^{i}\left(\mathbb{R}^{d}\right)$
the following function was defined by Benoist and Quint in \cite{MR2874934}.
Let 
\[
\varphi_{\epsilon}\left(v\right)=\begin{cases}
\min_{\lambda\in\mathcal{P}\setminus\left\{ 0\right\} }\epsilon^{\gamma_{i}}\left\Vert \tau_{\lambda}\left(v\right)\right\Vert ^{-1} & \textrm{ if }\left\Vert \tau_{0}\left(v\right)\right\Vert \leq\epsilon^{\gamma_{i}}\\
0 & \textrm{ else},
\end{cases}
\]
where for $0<i<d$ we define $\gamma_{i}=\left(d-i\right)i$. In fact,
the definition of $\varphi_{\epsilon}$ given here is a special case
of the definition given in \cite{MR2874934}. In the definition of
$\varphi_{\epsilon}$, given by Benoist and Quint, there is an extra
set of exponents depending on $\lambda\in\mathcal{P}\setminus\left\{ 0\right\} $
appearing. However, we see that in our case we may choose all of these
exponents to be equal to one. 

Let $\mathcal{F}=\left\{ v\in\bigwedge\left(\mathbb{R}^{d}\right):H_{I}v=v\right\} $
be the the fixed vectors of $H_{I}$. Let $\mathcal{F}^{c}$ be the
orthogonal complement of $\mathcal{F}$. We make the following remarks.
\begin{rem}
\label{norms}Since $\max_{\lambda\in\mathcal{P}\setminus\left\{ 0\right\} }\left\Vert \tau_{\lambda}\left(v\right)\right\Vert $
defines a norm on $\mathcal{F}^{c}$ there exists constants $c_{1}$
and $c_{2}$ depending on $\epsilon$ and the $\gamma_{i}'s$ such
that 
\[
c_{1}\left\Vert v\right\Vert ^{-1}\leq\varphi_{\epsilon}\left(v\right)\leq c_{2}\left\Vert v\right\Vert ^{-1}
\]
for all $v\in\mathcal{F}^{c}$. 
\end{rem}

\begin{rem}
\label{infity}For $0<i<d$ and $v\in\bigwedge^{i}\left(\mathbb{R}^{d}\right)\setminus\left\{ 0\right\} $
we have $\varphi_{\epsilon}\left(v\right)=\infty$ if and only if
$v$ is $H_{I}$ invariant and $\left\Vert v\right\Vert \leq\epsilon^{\gamma_{i}}$.
\end{rem}
We will need to refer to the constant defined as $b_{1}=\sup\left\{ \varphi_{\epsilon}\left(v\right):v\in\bigwedge\left(\mathbb{R}^{d}\right),\left\Vert v\right\Vert \geq1\right\} .$
In \cite{MR2874934} (Lemma 4.2) Benoist and Quint showed that the
function $\varphi_{\epsilon}$ satisfies the following convexity property. 
\begin{lem}
\label{lem: B+Q convexity}There exists a positive constant $C$ such
that for any $0<\epsilon<C^{-1}$, $u\in\Omega_{i_{1}}$, $v\in\Omega_{i_{2}}$
and $w\in\Omega_{i_{3}}$, with $i_{1}\geq0$, $i_{2}>0$ and $i_{3}>0$
such that $\varphi_{\epsilon}\left(u\wedge v\right)\geq1$ and $\varphi_{\epsilon}\left(u\wedge w\right)\geq1$,
one has:
\begin{enumerate}
\item \label{enu:B+Q, convexity 1}If $i_{1}>0$ and $i_{1}+i_{2}+i_{3}<d$,
then 
\[
\min\left\{ \varphi_{\epsilon}\left(u\wedge v\right),\varphi_{\epsilon}\left(u\wedge w\right)\right\} \leq\left(C\epsilon\right)^{1/2}\max\left\{ \varphi_{\epsilon}\left(u\right),\varphi_{\epsilon}\left(u\wedge v\wedge w\right)\right\} .
\]

\item \label{enu:B+Q, convexity 2}If $i_{1}=0$ and $i_{1}+i_{2}+i_{3}<d$,
then
\[
\min\left\{ \varphi_{\epsilon}\left(v\right),\varphi_{\epsilon}\left(w\right)\right\} \leq\left(C\epsilon\right)^{1/2}\varphi_{\epsilon}\left(v\wedge w\right).
\]

\item \label{enu:B+Q, convexity 3}If $i_{1}>0$, $i_{1}+i_{2}+i_{3}=d$
and $\left\Vert u\wedge v\wedge w\right\Vert \geq1$, then 
\[
\min\left\{ \varphi_{\epsilon}\left(u\wedge v\right),\varphi_{\epsilon}\left(u\wedge w\right)\right\} \leq\left(C\epsilon\right)^{1/2}\varphi_{\epsilon}\left(u\right).
\]

\item \label{enu:B+Q, convexity 4}If $i_{1}=0$, $i_{1}+i_{2}+i_{3}=d$
and $\left\Vert v\wedge w\right\Vert \geq1$, then
\[
\min\left\{ \varphi_{\epsilon}\left(v\right),\varphi_{\epsilon}\left(w\right)\right\} \leq b_{1}.
\]

\end{enumerate}
\end{lem}
We also need to obtain uniform bounds for the spherical averages of
$\varphi_{\epsilon}$. In order to do this we use the following Lemma
(Lemma 5.2) from \cite{MR1609447} will be used. 
\begin{lem}
\label{lem:5,2}Let $V$ be a finite-dimensional real inner product
space, $A$ a self-adjoint linear transformation of $V$, $K$ a closed
connected subgroup of $O\left(V\right)$, and $S$ a closed subset
of the unit sphere in $V$. Assume the only eigenvalues of $A$ are
$-1$, $0,$ and $1$ and denote by $W^{-}$, $W^{0}$ and $W^{+}$
the corresponding eigenspaces. Assume that $Kv\not\subset W^{0}$
for any $v\in S$ and that there exists a self-adjoint subgroup $H_{1}$
of $GL\left(V\right)$ with the following properties:
\begin{enumerate}
\item The Lie algebra of $H_{1}$ contains $A$.
\item $H_{1}$ is locally isomorphic to $SO\left(3,1\right)$.
\item $H_{1}\cap K$ is a maximal compact subgroup of $H_{1}$. 
\end{enumerate}
Then for any $\delta$, $0<\delta<2$, 
\[
\lim_{t\rightarrow\infty}\sup_{v\in S}\int_{K}\left\Vert \exp\left(tA\right)kv\right\Vert ^{-\delta}d\nu\left(k\right)=0.
\]

\end{lem}
Using Lemma \ref{lem:5,2} we can obtain the following bound on the
spherical averages.
\begin{lem}
\label{lem:r1>3r2>1}Suppose $r_{1}\geq3$ and $r_{2}\geq1$. Then
for all $\epsilon>0$, $0<\delta<2$ and $c>0$ there exists $t_{0}>0$
such that for all $t>t_{0}$ and all $v\in\mathcal{F}^{c}\setminus\left\{ 0\right\} $,
\[
\int_{K_{I}}\varphi_{\epsilon}\left(a_{t}kv\right)^{\delta}d\nu_{I}\left(k\right)<c\varphi_{\epsilon}\left(v\right)^{\delta}.
\]
\end{lem}
\begin{proof}
The subset $S=\left\{ v\in\bigwedge\left(\mathbb{R}^{d}\right):\left\Vert v-\tau_{0}\left(v\right)\right\Vert =1\right\} $
is a closed subset of the unit sphere in $\bigwedge\left(\mathbb{R}^{d}\right)$.
We have $a_{t}=\exp\left(tA\right)$, for an appropriate choice of
$A$ satisfying the conditions of Lemma \ref{lem:5,2}. 

We claim that for any $v\in S$, $Kv\not\subset W^{0}$. To see this,
let 
\[
H_{v}=\left\{ h\in H_{I}:hkv=kv\textrm{ for all }k\in K_{I}\right\} .
\]
Note that $K_{I}$ normalises $H_{v}$. Let $E_{v}$ be the subgroup
generated by $K_{I}\cup H_{v}$. By its definition $E_{v}$ also normalises
$H_{v}$. Since $K_{I}$ is a maximal proper subgroup of $H_{I}$,
in the case that $H_{v}\not\subset K_{I}$ we must have $E_{v}=H_{I}$.
Therefore, $H_{v}$ is a normal subgroup of $H_{I}$. Since $r_{1}\geq3$
and $r_{2}\geq1$, $H_{I}$ is simple and hence $H_{v}=H_{I}$ or
$H_{v}$ is trivial. Since $S\cap\mathcal{F}=0$, the first case is
impossible. Therefore, for all $v\in S$, $H_{v}\subset K_{I}$. In
particular this means that $\left\{ a_{t}:t\in\mathbb{R}\right\} $
is not contained in $H_{v}$. This implies the claim.

Then if $r_{1}\geq3$ and $r_{2}\geq1$ the conditions of Lemma \ref{lem:5,2}
are satisfied. Hence, for any $\delta$ with $0<\delta<2$, 
\[
\lim_{t\rightarrow\infty}\sup_{v\in S}\int_{K_{I}}\left\Vert a_{t}kv\right\Vert ^{-\delta}d\nu_{I}\left(k\right)=0.
\]
This implies that for all $c>0$, there exists $t_{0}>0$, such that
for all $t>t_{0}$ and all $v\in\mathcal{F}^{c}\setminus\left\{ 0\right\} $,
\[
\int_{K_{I}}\left\Vert a_{t}kv\right\Vert ^{-\delta}d\nu_{I}\left(k\right)<c\left\Vert v\right\Vert ^{-\delta}.
\]
Then the claim of the Lemma follows from Remark \ref{norms}.
\end{proof}

\subsubsection{A function on the space of lattices. }

For any lattice $\Lambda$, we say that $v\in\Omega$ is $\Lambda$-integral
if one can write $v=v_{1}\wedge\dots\wedge v_{i}$ where $v_{1},\dots,v_{i}\in\Lambda$.
Let $\Omega_{i}\left(\Lambda\right)$ and $\Omega\left(\Lambda\right)$
be the sets of $\Lambda$-integral elements of $\Omega_{i}$ and $\Omega$
respectively. Define $f_{\epsilon}:SL_{d}\left(\mathbb{R}\right)/SL_{d}\left(\mathbb{Z}\right)\to\mathbb{R}$
by 
\[
f_{\epsilon}\left(\Lambda\right)=\max_{v\in\Omega\left(\Lambda\right)}\varphi_{\epsilon}\left(v\right).
\]
Note that, by Remark \ref{norms} for all $\epsilon>0$ there exists
some constant $c_{\epsilon}>0$ such that for any unimodular lattice
$\Lambda$, we have 
\begin{alignat}{1}
\max_{v\in\Omega\left(\Lambda\right)}\left\Vert v\right\Vert ^{-1} & \leq\max_{0<i<d}\left(\max_{v\in\Omega_{i}\left(\Lambda\right),\left\Vert \tau_{0}\left(v\right)\right\Vert \leq\epsilon^{\gamma_{i}}}\left\Vert v\right\Vert ^{-1}+\max_{v\in\Omega_{i}\left(\Lambda\right),\left\Vert \tau_{0}\left(v\right)\right\Vert >\epsilon^{\gamma_{i}}}\left\Vert v\right\Vert ^{-1}\right)\nonumber \\
 & \leq c_{\epsilon}f_{\epsilon}\left(\Lambda\right)+\max_{0<i<d}\epsilon^{-\gamma_{i}}.\label{eq:property 5}
\end{alignat}
Moreover, it follows from the definition of the $\alpha$ function
that 
\begin{equation}
\alpha\left(\Lambda\right)=\max\left\{ \max_{v\in\Omega\left(\Lambda\right)}\left\Vert v\right\Vert ^{-1},1\right\} .\label{eq:forgot}
\end{equation}
The following Lemma is necessary to ensure that the function $f_{\epsilon}\left(h\Delta\right)$
is finite for all $h\in H_{I}$. 
\begin{lem}
\label{lem:fixed vectors}For all $h\in H_{I}$, if $u\in\Omega\left(h\Delta\right)$,
then $u\not\in\mathcal{F}$.\end{lem}
\begin{proof}
Suppose for a contradiction that $u\in\Omega\left(h\Delta\right)\cap\mathcal{F}$.
Suppose that $u$ has degree $i$ for some $0<i<d$ and let $u=u_{1}\wedge\dots\wedge u_{i}$
and $U=\left\langle u_{1},\dots,u_{i}\right\rangle $. Since $u\in\Omega\left(h\Delta\right)$,
it follows that $U\cap h\Delta$ is a lattice in $U$. Moreover, because
$u\in\mathcal{F}$, $U\cap\Delta$ is also a lattice in $U$, or equivalently
$g^{-1}U\cap\mathbb{Z}^{d}$ is a lattice in $g^{-1}U$. The subspace
$g^{-1}U$ is $H_{g}$ invariant. 

Conversely, it follows from Lemma 3.4 of \cite{2011arXiv1111.4428S}
that if $V$ is any $H_{g}$ invariant subspace, then either
\begin{enumerate}
\item $V\subseteq g^{-1}\left\langle e_{1},\dots,e_{s}\right\rangle $ or,
\item $V=$\foreignlanguage{english}{$g^{-1}\left\langle e_{s+1},\dots,e_{d}\right\rangle \oplus V'$}
where $V'\subseteq g^{-1}\left\langle e_{1},\dots,e_{s}\right\rangle .$
\end{enumerate}
Therefore, either $V$ or the orthogonal complement of $V$ is contained
in $g^{-1}\left\langle e_{1},\dots,e_{s}\right\rangle $. By Corollary
3.2 of \cite{2011arXiv1111.4428S}, $g^{-1}\left\langle e_{1},\dots,e_{s}\right\rangle $
contains no subspaces defined over $\mathbb{Q}$. This implies that
if $V$ is any $H_{g}$ invariant subspace then $V$ is not defined
over $\mathbb{Q}$. In particular $V\cap\mathbb{Z}^{d}$ cannot be
a lattice in $V$. This gives a contradiction. 
\end{proof}

\subsubsection{A function on $H_{I}$.}

Define $\widetilde{f}_{\Delta,\epsilon}:H_{I}\rightarrow\mathbb{R}$
by 
\[
\widetilde{f}_{\Delta,\epsilon}\left(h\right)=f_{\epsilon}\left(h\Delta\right).
\]
In view of \eqref{eq:property 5} and \eqref{eq:forgot} the proof
of part I of Theorem \ref{thm:upper bound-1} will be complete provided
that that the conditions \eqref{enu:contraction 1}-\eqref{enu:contraction 4}
from Proposition \ref{prop:Let--be} are satisfied by the function
$\widetilde{f}_{\Delta,\epsilon}$ for some $\epsilon>0$. It is clear
that \foreignlanguage{english}{$\widetilde{f}_{\Delta,\epsilon}$}
is left $K_{I}$ invariant. Also, since $\left\Vert \tau_{\lambda}\left(\rho\left(h^{-1}\right)\right)\right\Vert ^{-1}\leq\left\Vert \tau_{\lambda}\left(hv\right)\right\Vert /\left\Vert v\right\Vert \leq\left\Vert \tau_{\lambda}\left(\rho\left(h\right)\right)\right\Vert $
for all $\lambda\in\mathcal{P}$, $v\in\Omega$ and $h\in H_{I}$,
$\widetilde{f}_{\Delta,\epsilon}$ also satisfies condition \eqref{enu:contraction 1}
of Proposition \ref{prop:Let--be}. From Remark \ref{infity} we get
that $\widetilde{f}_{\Delta,\epsilon}\left(1\right)=\infty$ only
if there exists $v\in\Omega\left(\Delta\right)\cap\mathcal{F}$, but
by Lemma \ref{lem:fixed vectors} we know that no such $v$ exists
and so $\widetilde{f}_{\Delta,\epsilon}\left(1\right)<\infty$. It
remains to show that $\widetilde{f}_{\Delta,\epsilon}$ is contracted
by the operator $A_{t}$. The proof is very similar to that of Proposition
5.3 in \cite{MR2874934}.
\begin{lem}
Suppose $r_{1}\geq3$ and $r_{2}\geq1$. There exists $\epsilon>0$
such that for all $0<\delta<2$, the function $\widetilde{f}_{\Delta,\epsilon}^{\delta}$
is contracted by the operator $A_{t}$. \end{lem}
\begin{proof}
Fix $c>0$. By Lemma \ref{lem:r1>3r2>1} there exists $t_{0}>0$,
so that for any $v\in\mathcal{F}^{c}\setminus\left\{ 0\right\} $,
\begin{equation}
\int_{K_{I}}\varphi_{\epsilon}\left(a_{t_{0}}kv\right)^{\delta}d\nu_{I}\left(k\right)<\frac{c}{d}\varphi_{\epsilon}\left(v\right)^{\delta}.\label{eq:contraction 1}
\end{equation}
Let $m_{0}=\left\Vert \rho\left(a_{t_{0}}\right)\right\Vert =\left\Vert \rho\left(a_{t_{0}}^{-1}\right)\right\Vert $.
Then, for all $v\in\bigwedge\left(\mathbb{R}^{d}\right)$, 
\begin{equation}
m{}_{0}^{-1}\leq\left\Vert a_{t_{0}}v\right\Vert /\left\Vert v\right\Vert \leq m_{0}.\label{eq:ineq2}
\end{equation}
It follows from the definition of $\varphi_{\epsilon}$ and \eqref{eq:ineq2}
that 
\begin{equation}
m{}_{0}^{-1}\varphi_{\epsilon}\left(v\right)\leq\varphi_{\epsilon}\left(a_{t_{0}}v\right)\leq m_{0}\varphi_{\epsilon}\left(v\right).\label{eq:contraction 2}
\end{equation}
Let 
\[
\Psi\left(h\Delta\right)=\left\{ v\in\Omega\left(h\Delta\right):f_{\epsilon}\left(h\Delta\right)\leq m_{0}^{2}\varphi_{\epsilon}\left(v\right)\right\} .
\]
Note that 
\begin{equation}
f_{\epsilon}\left(h\Delta\right)=\max_{\psi\in\Psi\left(h\Delta\right)}\varphi_{\epsilon}\left(\psi\right).\label{eq:contraction 3}
\end{equation}
Let $C$ be the constant from Lemma \ref{lem: B+Q convexity}. Assume
that $\epsilon$ is small enough so that 
\begin{equation}
m_{0}^{4}C\epsilon<1.\label{eq:contradiction}
\end{equation}
There are now two cases.\smallskip{}

\noindent \emph{Case 1}. If $f_{\epsilon}\left(h\Delta\right)\leq\max\left\{ b_{1},m_{0}^{2}\right\} $.
In this case \eqref{eq:contraction 2} and the fact that $f_{\epsilon}$
is left $K_{I}$ invariant, imply that $f_{\epsilon}\left(a_{t_{0}}kh\Delta\right)\leq m_{0}f_{\epsilon}\left(h\Delta\right)$.
Hence 
\begin{equation}
\int_{K_{I}}f_{\epsilon}\left(a_{t_{0}}kh\Delta\right)^{\delta}d\nu_{I}\left(k\right)\leq\left(m_{0}\max\left\{ b_{1},m_{0}^{2}\right\} \right)^{\delta}.\label{eq:case 1}
\end{equation}
\smallskip{}

\noindent \emph{Case 2.} If $f_{\epsilon}\left(h\Delta\right)>\max\left\{ b_{1},m_{0}^{2}\right\} $.
This implies:
\begin{claim}
\label{claim}The set $\Psi\left(h\Delta\right)$ contains only one
element up to sign change in each degree. 
\end{claim}

We verify the claim as follows. Assume that for some $0<i<d$, $\Psi\left(h\Delta\right)\cap\Omega\left(h\Delta\right)$
contains two non-colinear elements, $v_{0}$ and $w_{0}$. Then because
$f_{\epsilon}\left(h\Delta\right)>m_{0}^{2}$ and $v_{0}$ and $w_{0}$
are in $\Psi\left(h\Delta\right)$, we have $\varphi_{\epsilon}\left(v_{0}\right)\geq1$
and \foreignlanguage{english}{$\varphi_{\epsilon}\left(w_{0}\right)\geq1$}.
We can write $v_{0}=u\wedge v$ and $w_{0}=u\wedge w$ where $u\in\Omega_{i_{1}}\left(h\Delta\right),$
$v\in\Omega_{i_{2}}\left(h\Delta\right)$ and $w\in\Omega_{i_{2}}\left(h\Delta\right)$
with $i_{1}\geq0$ and $i_{2}>0$. There are four cases.\smallskip{}

\noindent \emph{Case 2.1.} If $i_{2}<i$ and $i_{2}<d-i$. In this
case
\[
f_{\epsilon}\left(h\Delta\right)\leq m_{0}^{2}\min\left\{ \varphi_{\epsilon}\left(u\wedge v\right),\varphi_{\epsilon}\left(u\wedge w\right)\right\} \leq\left(m_{0}^{4}C\epsilon\right)^{1/2}\max\left\{ \varphi_{\epsilon}\left(u\right),\varphi_{\epsilon}\left(u\wedge v\wedge w\right)\right\} ,
\]
 by Lemma \ref{lem: B+Q convexity} part \eqref{enu:B+Q, convexity 1}.
This implies that 
\begin{equation}
f_{\epsilon}\left(h\Delta\right)\leq\left(m_{0}^{4}C\epsilon\right)^{1/2}f_{\epsilon}\left(h\Delta\right),\label{eq:contradiction 2}
\end{equation}
which contradicts \eqref{eq:contradiction}.\smallskip{}

\noindent \emph{Case 2.2}. If $i_{2}=i<d-i$. In this case $u=1$.
The same computation but using Lemma \ref{lem: B+Q convexity} part
\eqref{enu:B+Q, convexity 2} still gives \eqref{eq:contradiction 2}
which is still a contradiction. \smallskip{}

\noindent \emph{Case 2.3.} If $i_{2}=d-i<i$. In this case $\left\Vert u\wedge v\wedge w\right\Vert $
is an integer. Therefore, the same computation but using Lemma \ref{lem: B+Q convexity}
part \eqref{enu:B+Q, convexity 3} still gives \eqref{eq:contradiction 2}.\smallskip{}

\noindent \emph{Case 2.4.} If $i_{2}=i=d-i$. The same computation,
using Lemma \ref{lem: B+Q convexity} part \eqref{enu:B+Q, convexity 4}
gives 
\[
f_{\epsilon}\left(h\Delta\right)\leq b_{1},
\]
which is again a contradiction. \smallskip{}

\noindent This completes the proof of the claim.

\medskip{}
Suppose $v\in\Omega$ is arbitrary. If $v\not\in\Psi\left(h\Delta\right)$,
then $f_{\epsilon}\left(h\Delta\right)>m_{0}^{2}\varphi_{\epsilon}\left(v\right),$
and by left $K_{I}$ invariance of $\varphi_{\epsilon}$, \eqref{eq:contraction 2}
and \eqref{eq:contraction 3} for all $k\in K_{I}$ we have 
\begin{equation}
\varphi_{\epsilon}\left(a_{t_{0}}kv\right)\leq m_{0}\varphi_{\epsilon}\left(v\right)\leq m_{0}^{-1}f_{\epsilon}\left(h\Delta\right)\leq m_{0}^{-1}\max_{\psi\in\Psi\left(h\Delta\right)}\varphi_{\epsilon}\left(\psi\right)\leq\max_{\psi\in\Psi\left(h\Delta\right)}\varphi_{\epsilon}\left(a_{t_{0}}k\psi\right).\label{eq:contraction 4}
\end{equation}
If $v\in\Psi\left(h\Delta\right)$, then \eqref{eq:contraction 4}
holds for obvious reasons. Therefore \eqref{eq:contraction 4} holds
for all $v\in\Omega$. Thus using the definition of $f_{\epsilon}$
and \eqref{eq:contraction 4} we get 
\begin{alignat}{1}
\int_{K_{I}}f_{\epsilon}\left(a_{t_{0}}kh\Delta\right)^{\delta}d\nu_{I}\left(k\right) & =\int_{K_{I}}\max_{v\in\Omega\left(h\Delta\right)}\varphi_{\epsilon}\left(a_{t_{0}}kv\right)^{\delta}d\nu_{I}\left(k\right)\leq\sum_{\psi\in\Psi\left(h\Delta\right)}\int_{K_{I}}\varphi_{\epsilon}\left(a_{t_{0}}k\psi\right)^{\delta}d\nu_{I}\left(k\right).\label{eq:contraction 5}
\end{alignat}
Using Lemma \ref{lem:fixed vectors} we see that for all $\psi\in\Psi\left(h\Delta\right)$,
$\psi\not\in\mathcal{F}$ and hence $\psi-\tau_{0}\left(\psi\right)\in\mathcal{F}^{c}\setminus\left\{ 0\right\} $.
Moreover, if $\varphi_{\epsilon}\left(a_{t_{0}}k\psi\right)\neq0$,
then $\varphi_{\epsilon}\left(a_{t_{0}}k\psi\right)=\varphi_{\epsilon}\left(a_{t_{0}}k\left(\psi-\tau_{0}\left(\psi\right)\right)\right)$
and we can apply \eqref{eq:contraction 1} to get 
\begin{equation}
\int_{K_{I}}\varphi_{\epsilon}\left(a_{t_{0}}k\psi\right)^{\delta}d\nu_{I}\left(k\right)\leq\frac{c}{d}\varphi_{\epsilon}\left(\psi\right)^{\delta},\label{eq:contarction 6}
\end{equation}
 for each $\psi\in\Psi\left(h\Delta\right)$. If $\varphi_{\epsilon}\left(a_{t_{0}}k\psi\right)=0$,
then it is clear that \eqref{eq:contarction 6} also holds. Using
Claim \ref{claim}, we obtain 
\[
\sum_{\psi\in\Psi\left(h\Delta\right)}\int_{K_{I}}\varphi_{\epsilon}\left(a_{t_{0}}k\psi\right)^{\delta}d\nu_{I}\left(k\right)\leq d\max_{\psi\in\Psi\left(h\Delta\right)}\int_{K_{I}}\varphi_{\epsilon}\left(a_{t_{0}}k\psi\right)^{\delta}d\nu_{I}\left(k\right),
\]
the claim of the Lemma follows from \eqref{eq:contraction 3}, \eqref{eq:case 1},
\eqref{eq:contraction 5} and \eqref{eq:contarction 6}. 
\end{proof}

\subsection{Proof of part II of Theorem \ref{thm:upper bound-1}. }

This time the aim is to construct a function such that it satisfies
the conditions of the following Lemma. 
\begin{lem}
\label{lem:SL(2,R)}Suppose $r_{1}=2$ and $r_{2}=1$ or $r_{1}=r_{2}=2$.
Let $f:H_{I}\to\mathbb{R}$ be a strictly positive continuous function
such that:
\begin{enumerate}
\item \label{enu:contraction 1-1}For any $\epsilon>0$, there exists a
neighbourhood $V\left(\epsilon\right)$ of $1$ in $H_{I}$ such that
\[
\left(1-\epsilon\right)f\left(h\right)\leq f\left(uh\right)\leq\left(1+\epsilon\right)f\left(h\right)
\]
for all $h\in H_{I}$ and $u\in V\left(\epsilon\right)$.
\item \label{enu:contraction 2-1}The function $f$ is left $K_{I}$ invariant. 
\item \label{enu:contraction 3-1}$f\left(1\right)<\infty$. 
\item \label{enu:contraction 4-1}There exists $t_{0}>0$ and $b>0$, such
that for all $h\in H_{I}$ and $0\leq t\leq t_{0}$, 
\[
A_{t}f\left(h\right)\leq f\left(h\right)+b.
\]

\end{enumerate}
Then $\sup_{t>1}\frac{1}{t}A_{t}f\left(1\right)<\infty$.\end{lem}
\begin{proof}
Since $SO\left(2,1\right)$ is locally isomorphic to $SL_{2}\left(\mathbb{R}\right)$
and $SO\left(2,2\right)$ is locally isomorphic to $SL_{2}\left(\mathbb{R}\right)\times SL_{2}\left(\mathbb{R}\right)$,
this follows directly from Lemma 5.13 from \cite{MR1609447}. 
\end{proof}
The general strategy of this subsection is broadly the same as in
the last one. First we define a certain function on the exterior algebra
of $\mathbb{R}^{d}$ and then we use this function to define a function
which has the properties demanded by Lemma \ref{lem:SL(2,R)}.

\subsubsection{Functions on the exterior algebra of $\mathbb{R}^{d}$. }

As before we work with a function on the exterior algebra of $\mathbb{R}^{d}$.
This time the definition is simpler because in this case the vectors
fixed by the action of $H_{I}$ cause no extra problems. For $\epsilon>0$,
$0<i<d$ and $v\in\bigwedge^{i}\left(\mathbb{R}^{d}\right)$ we define
\[
\widetilde{\varphi}_{\epsilon}\left(v\right)=\epsilon^{\gamma_{i}}\left\Vert v\right\Vert ^{-1}.
\]
If $v\in\bigwedge^{0}\left(\mathbb{R}^{d}\right)$ or $v\in\bigwedge^{d}\left(\mathbb{R}^{d}\right)$
then we set $\widetilde{\varphi}_{\epsilon}\left(v\right)=1$. The
following Lemma is the analogue of Lemma \ref{lem: B+Q convexity}. 
\begin{lem}
\label{lem:simple convexity}Let $i_{1}\geq0$ and $i_{2}>0$ and
$\Lambda$ be a unimodular lattice . Then for all $u\in\Omega_{i_{1}}\left(\Lambda\right),$$v\in\Omega_{i_{2}}\left(\Lambda\right)$
and $w\in\Omega_{i_{2}}\left(\Lambda\right)$, 
\[
\widetilde{\varphi}_{\epsilon}\left(u\wedge v\right)\widetilde{\varphi}_{\epsilon}\left(u\wedge w\right)\leq\epsilon^{2i_{2}}\widetilde{\varphi}_{\epsilon}\left(u\right)\widetilde{\varphi}_{\epsilon}\left(u\wedge v\wedge w\right).
\]
\end{lem}
\begin{proof}
This is a direct consequence of Lemma 5.6 from \cite{MR1609447} and
the fact that $2\gamma_{i_{1}+i_{2}}-\gamma_{i_{1}}-\gamma_{i_{1}+2i_{2}}=2i_{2}$.
\end{proof}
The following Lemma is used to bound the spherical averages. It is
analogous to Lemma \ref{lem:r1>3r2>1}, (see also Lemma 5.5 from \cite{MR1609447}).
It explains why in this case the fixed vectors do not cause problems. 
\begin{lem}
\label{lem:derivative bound}Suppose $r_{1}\geq2$ and $r_{2}\geq1$.
Then for all $t\geq0$ and $v\in\bigwedge\left(\mathbb{R}^{d}\right)\setminus\left\{ 0\right\} $,
\[
\int_{K_{I}}\left\Vert a_{t}kv\right\Vert ^{-1}d\nu_{I}\left(k\right)\leq\left\Vert v\right\Vert ^{-1}.
\]
\end{lem}
\begin{proof}
Let $F_{v}\left(t\right)=\int_{K_{I}}\left\Vert a_{t}kv\right\Vert ^{-1}d\nu_{I}\left(k\right)$.
We will show that $\frac{d}{dt}F_{v}\left(t\right)\leq0$ for all
$t\geq0$ and $v\in\bigwedge\left(\mathbb{R}^{d}\right)\setminus\left\{ 0\right\} $,
from which it is clear that the claim of the Lemma follows. Let $\pi^{-}$
and $\pi^{+}$ be the projections from $\bigwedge\left(\mathbb{R}^{d}\right)$
onto the contracting and expanding eigenspaces of $a_{t}$ respectively.
Note that 
\begin{alignat}{1}
\frac{d}{dt}F_{v}\left(t\right) & =\int_{K_{I}}\frac{e^{-2t}\left\Vert \pi^{-}\left(kv\right)\right\Vert ^{2}-e^{2t}\left\Vert \pi^{+}\left(kv\right)\right\Vert ^{2}}{\left\Vert a_{t}kv\right\Vert ^{3}}d\nu_{I}\left(k\right)\nonumber \\
 & \leq\left(\frac{\left\Vert a_{t}\right\Vert }{\left\Vert v\right\Vert }\right)^{3}\int_{K_{I}}e^{-2t}\left\Vert \pi^{-}\left(kv\right)\right\Vert ^{2}-e^{2t}\left\Vert \pi^{+}\left(kv\right)\right\Vert ^{2}d\nu_{I}\left(k\right).\label{eq:new proof 1}
\end{alignat}
Let $Q_{0}$ also denote the matrix that defines the quadratic form
$Q_{0}$. Note that $\left\Vert \pi^{-}\left(Q_{0}v\right)\right\Vert =\left\Vert \pi^{+}\left(v\right)\right\Vert $
and $\left\Vert \pi^{+}\left(Q_{0}v\right)\right\Vert =\left\Vert \pi^{-}\left(v\right)\right\Vert $
for all $v\in\bigwedge\left(\mathbb{R}^{d}\right)$. Because $Q_{0}^{T}=Q_{0}=Q_{0}^{-1}$,
if $\det\left(Q_{0}\right)=1$ then $Q_{0}\in K_{I}$ or if $\det\left(Q_{0}\right)=-1$
then $-Q_{0}\in K_{I}$. This means that $Q_{0}K_{I}\left(v-\tau_{0}\left(v\right)\right)=K_{I}\pm\left(v-\tau_{0}\left(v\right)\right)$
and thus
\begin{equation}
\int_{K_{I}}\left\Vert \pi^{-}\left(kv\right)\right\Vert ^{2}d\nu_{I}\left(k\right)=\int_{K_{I}}\left\Vert \pi^{+}\left(Q_{0}\left(kv\right)\right)\right\Vert ^{2}d\nu_{I}\left(k\right)=\int_{K_{I}}\left\Vert \pi^{+}\left(kv\right)\right\Vert ^{2}d\nu_{I}\left(k\right).\label{eq:new proof 2}
\end{equation}
Therefore, using \eqref{eq:new proof 1} and \eqref{eq:new proof 2}
we have 
\[
\frac{d}{dt}F_{v}\left(t\right)\leq\left(\frac{\left\Vert a_{t}\right\Vert }{\left\Vert v\right\Vert }\right)^{3}\int_{K_{I}}\left\Vert \pi^{+}\left(kv\right)\right\Vert ^{2}d\nu_{I}\left(k\right)\left(e^{-2t}-e^{2t}\right)\leq0
\]
for all $t\geq0$ and $v\in\bigwedge\left(\mathbb{R}^{d}\right)\setminus\left\{ 0\right\} $
as required. 
\end{proof}

\subsubsection{Functions on $H_{I}$.}

Define $\widetilde{f}_{\Delta,\epsilon}:H_{I}\rightarrow\mathbb{R}$
by 
\[
\widetilde{f}_{\Delta,\epsilon}\left(h\right)=\sum_{i=1}^{d}\max_{v\in\Omega_{i}\left(h\Delta\right)}\widetilde{\varphi}_{\epsilon}\left(v\right).
\]
Note that for all $\epsilon>0$ there exists some constant $c_{\epsilon}>0$
such that for any unimodular lattice $\Lambda$, 
\[
\max_{v\in\Omega\left(\Lambda\right)}\left\Vert v\right\Vert ^{-1}\leq c_{\epsilon}\max_{v\in\Omega\left(\Lambda\right)}\widetilde{\varphi}_{\epsilon}\left(v\right)\leq c_{\epsilon}\sum_{i=1}^{d}\max_{v\in\Omega_{i}\left(\Lambda\right)}\widetilde{\varphi}_{\epsilon}\left(v\right).
\]
In view of this and \eqref{eq:forgot}, the proof of part II of Theorem
\ref{thm:upper bound-1} will be complete provided that that the conditions
\eqref{enu:contraction 1-1}-\eqref{enu:contraction 4-1} from Proposition
\ref{lem:SL(2,R)} are satisfied by the functions $\widetilde{f}_{\Delta,\epsilon}$
for some $\epsilon>0$. It is clear that \foreignlanguage{english}{$\widetilde{f}_{\Delta,\epsilon}$}
is left $K_{I}$ invariant. Also, since $\left\Vert \rho\left(h^{-1}\right)\right\Vert ^{-1}\leq\left\Vert hv\right\Vert /\left\Vert v\right\Vert \leq\left\Vert \rho\left(h\right)\right\Vert $
for all $v\in\Omega$ and $h\in H_{I}$, $\widetilde{f}_{\Delta,\epsilon}$
also satisfies condition \eqref{enu:contraction 1-1} of Proposition
\ref{lem:SL(2,R)}. We also have that $\widetilde{f}_{\Delta,\epsilon}\left(1\right)<\infty$.
It remains to show that $\widetilde{f}_{\Delta,\epsilon}$ satisfies
condition \eqref{enu:contraction 4-1} of Proposition \ref{lem:SL(2,R)}. 
\begin{lem}
Suppose $r_{1}=2$ and $r_{2}=1$ or $r_{1}=r_{2}=2$. Then there
exists $\epsilon>0$ and $t_{0}>0$, such that for all $0\leq t<t_{0}$
and $h\in H_{I}$,
\[
\int_{K_{I}}\widetilde{f}_{\Delta,\epsilon}\left(a_{t}kh\right)d\nu_{I}\left(k\right)\leq\widetilde{f}_{\Delta,\epsilon}\left(h\right).
\]
\end{lem}
\begin{proof}
Let $m_{0}=\left\Vert \rho\left(a_{t_{0}}\right)\right\Vert $. Then,
for all $v\in\bigwedge\left(\mathbb{R}^{d}\right)$ and $0\leq t<t_{0}$,
\begin{equation}
m{}_{0}^{-1}\leq\left\Vert a_{t}v\right\Vert /\left\Vert v\right\Vert \leq m_{0}.\label{eq:ineq2-1}
\end{equation}
It follows from the definition of $\widetilde{\varphi}_{\epsilon}$
and \eqref{eq:ineq2-1} that for all $0\leq t<t_{0}$, 
\begin{equation}
m{}_{0}^{-1}\widetilde{\varphi}_{\epsilon}\left(v\right)\leq\widetilde{\varphi}_{\epsilon}\left(a_{t}v\right)\leq m_{0}\widetilde{\varphi}_{\epsilon}\left(v\right).\label{eq:contraction 2-1}
\end{equation}
Let 
\[
\Psi\left(h\Delta\right)=\bigcup_{i=1}^{d}\left\{ v\in\Omega_{i}\left(h\Delta\right):\max_{v\in\Omega_{i}\left(h\Delta\right)}\widetilde{\varphi}_{\epsilon}\left(v\right)\leq m_{0}^{2}\widetilde{\varphi}_{\epsilon}\left(v\right)\right\} .
\]
Now we show that: For $\epsilon$ small enough the set $\Psi\left(h\Delta\right)$
contains only one element up to sign change in each degree. To see
this, assume that for some $0<i<d$, $\Psi\left(h\Delta\right)\cap\Omega\left(h\Delta\right)$
contains two non-colinear elements, $v_{0}$ and $w_{0}$. We can
write $v_{0}=u\wedge v$ and $w_{0}=u\wedge w$ where $u\in\Omega_{i_{1}}\left(h\Delta\right),$$v\in\Omega_{i_{2}}\left(h\Delta\right)$
and $w\in\Omega_{i_{2}}\left(h\Delta\right)$ with $i_{1}\geq0$ and
$i_{2}>0$. In this case
\[
\widetilde{f}_{\Delta,\epsilon}\left(h\right)^{2}\leq d^{2}m_{0}^{4}\widetilde{\varphi}_{\epsilon}\left(u\wedge v\right)\widetilde{\varphi}_{\epsilon}\left(u\wedge w\right)\leq d^{2}m_{0}^{4}\epsilon^{2i_{2}}\widetilde{f}_{\Delta,\epsilon}\left(h\right)^{2},
\]
 by Lemma \ref{lem:simple convexity}. Hence the claim is true since
taking $\epsilon$ small enough gives a contradiction.

In view of this discussion we can suppose that $\Psi\left(h\Delta\right)=\left\{ \psi_{i}\right\} _{i=1}^{d}$
where $\psi_{i}$ has degree $i$. Let $v\in\Omega_{i}\left(h\Delta\right)$
be arbitrary. If $v\not\in\Psi\left(h\Delta\right)$, then $\max_{v\in\Omega_{i}\left(h\Delta\right)}\widetilde{\varphi}_{\epsilon}\left(v\right)>m_{0}^{2}\widetilde{\varphi}_{\epsilon}\left(v\right),$
and by left $K_{I}$ invariance of $\widetilde{\varphi}_{\epsilon}$
and \eqref{eq:contraction 2-1} for all $k\in K_{I}$ we have 
\begin{equation}
\widetilde{\varphi}_{\epsilon}\left(a_{t_{0}}kv\right)\leq m_{0}\widetilde{\varphi}_{\epsilon}\left(v\right)\leq m_{0}^{-1}\max_{v\in\Omega_{i}\left(h\Delta\right)}\widetilde{\varphi}_{\epsilon}\left(v\right)=m_{0}^{-1}\widetilde{\varphi}_{\epsilon}\left(\psi_{i}\right)\leq\widetilde{\varphi}_{\epsilon}\left(a_{t_{0}}k\psi_{i}\right).\label{eq:contraction 4-1}
\end{equation}
If $v\in\Psi\left(h\Delta\right)$, then \eqref{eq:contraction 4-1}
holds for obvious reasons. Therefore \eqref{eq:contraction 4-1} holds
for all $v\in\Omega$. Thus using the definition of $\widetilde{f}_{\Delta,\epsilon}$
and \eqref{eq:contraction 4-1} we get 
\begin{alignat}{1}
\int_{K_{I}}\widetilde{f}_{\Delta,\epsilon}\left(a_{t_{0}}kh\right)d\nu_{I}\left(k\right) & =\sum_{i=1}^{d}\int_{K_{I}}\max_{v\in\Omega_{i}\left(h\Delta\right)}\widetilde{\varphi}_{\epsilon}\left(a_{t_{0}}kv\right)d\nu_{I}\left(k\right)\leq\sum_{i=1}^{d}\int_{K_{I}}\widetilde{\varphi}_{\epsilon}\left(a_{t_{0}}k\psi_{i}\right)d\nu_{I}\left(k\right).\label{eq:contraction 5-1}
\end{alignat}
By Lemma \ref{lem:derivative bound} there exists $t_{0}>0$, so that
for any $v\in\bigwedge\left(\mathbb{R}^{d}\right)$ and all $0\leq t<t_{0}$,
\begin{equation}
\int_{K_{I}}\widetilde{\varphi}_{\epsilon}\left(a_{t_{0}}k\psi_{i}\right)d\nu_{I}\left(k\right)\leq\widetilde{\varphi}_{\epsilon}\left(\psi_{i}\right),\label{eq:contarction 6-1}
\end{equation}
 for each $\psi_{i}\in\Psi\left(h\Delta\right)$. The claim of the
Lemma follows from \eqref{eq:contraction 5-1} and \eqref{eq:contarction 6-1}. 
\end{proof}

\section{Ergodic Theorems.\label{sec:Ergodic-theorems.}}

For subgroups $W_{1}$ and $W_{2}$ of $G_{g}$, let $X\left(W_{1},W_{2}\right)=\left\{ g\in G_{g}:W_{2}g\subset gW_{1}\right\} $.
As in \cite{MR1609447} the ergodic theory is based on Theorem 3 from
\cite{MR1237827} reproduced below in a form relevant to the current
situation. 
\begin{thm}
\label{thm:dani and margulis}Suppose $r_{1}\geq2$ and $r_{2}\geq1$.
Let $g\in\mathcal{C}_{SL}\left(r_{1},r_{2}\right)$ be arbitrary.
Let $U=\left\{ u_{t}:t\in\mathbb{R}\right\} $ be a unipotent one
parameter subgroup of $G_{g}$ and $\phi$ be a bounded continuous
function on $G_{g}/\Gamma_{g}$. Let $\mathcal{D}$ be a compact subset
of $G_{g}/\Gamma_{g}$ and let $\epsilon>0$ be given. Then there
exist finitely many proper closed subgroups $H_{1},\ldots,H_{k}$
of $G_{g}$, such that $H_{i}\cap\Gamma_{g}$ is a lattice in $H_{i}$
for all $i$, and compact subsets $C_{1},\ldots,C_{k}$ of $X\left(H_{1},U\right),\ldots,X\left(H_{k},U\right)$
respectively such that for all compact subsets $F$ of $\mathcal{D}-\bigcup_{1\leq i\leq k}C_{i}\Gamma_{g}/\Gamma_{g}$
there exists a $T_{0}>0$ such that for all $x\in F$ and $T>T_{0}$,
\[
\left|\frac{1}{T}\int_{0}^{T}\phi\left(u_{t}x\right)dt-\int_{G_{g}/\Gamma_{g}}\phi d\mu_{g}\right|<\epsilon.
\]
\end{thm}
\begin{rem}
\label{rational subgroups}By construction the subgroups $H_{i}$
occurring are such that $H_{i}\cap\Gamma_{g}$ is Zariski dense in
$H_{i}$ and hence $H_{i}$ are defined over $\mathbb{Q}$. For a
precise reference see Theorem 3.6.2 and Remark 3.4.2 of \cite{MR1928528}.
\end{rem}
The next result is a reworking of Theorem 4.3 from \cite{MR1609447}.
The difference is that in Lemma \ref{thm:4.4} the identity is fixed
as the base point for the flow and the condition that $H_{g}$ be
maximal is dropped. 
\begin{lem}
\label{thm:4.4}Suppose $r_{1}\geq2$ and $r_{2}\geq1$. Let $g\in\mathcal{C}_{SL}\left(r_{1},r_{2}\right)$
be arbitrary. Let $U=\left\{ u_{t}:t\in\mathbb{R}\right\} $ be a
one parameter unipotent subgroup of $H_{g}$, not contained in any
proper normal subgroup of $H_{g}$. Let $\phi$ be a bounded continuous
function on $G_{g}/\Gamma_{g}$. Then for all $\epsilon>0$ and $\eta>0$
there exists a $T_{0}>0$ such that for all $T>T_{0}$, 
\begin{equation}
\nu_{g}\left(\left\{ k\in K_{g}:\left|\frac{1}{T}\int_{0}^{T}\phi\left(u_{t}k\right)dt-\int_{G_{g}/\Gamma_{g}}\phi d\mu_{g}\right|>\epsilon\right\} \right)\leq\eta.\label{eq:erg conclusion 1}
\end{equation}
\end{lem}
\begin{proof}
Let $H_{1},\ldots,H_{k}$ and $C_{1},\ldots,C_{k}$ be as in Theorem
\ref{thm:dani and margulis}. Let $\gamma\in\Gamma_{g}$, consider
$Y_{i}\left(\gamma\right)=K_{g}\cap X\left(H_{i},U\right)\gamma$.
Suppose that $Y_{i}\left(\gamma\right)=K_{g}$, then $Uk\gamma^{-1}\subset k\gamma^{-1}H_{i}$
for all $k\in K_{g}$. In other words
\begin{equation}
k^{-1}Uk\subset\gamma^{-1}H_{i}\gamma\quad\textrm{ for all }k\in K_{g}.\label{eq:erg1}
\end{equation}
The subgroup $\left\langle k^{-1}Uk:k\in K_{g}\right\rangle $ is
normalised by $U\cup K_{g}$ and clearly $\left\langle k^{-1}Uk:k\in K_{g}\right\rangle \subseteq\left\langle U\cup K_{g}\right\rangle \subseteq H_{g}$.
If $G$ is a simple Lie group with finite centre, with maximal compact
subgroup $K$, it follows from exercise A.3, chapter IV of \cite{MR1834454}
that $K$ is also a maximal proper subgroup of $G$. This means that
because $H_{g}$ is semisimple with finite centre, any connected subgroup
$L$ of $H_{g}$ containing $K_{g}$ can be represented as $L=H'K_{g}$
where $H'$ is a connected normal subgroup of $H_{g}$. Because $U$
is not contained in any proper normal subgroup of $H_{g}$, this implies
that $\left\langle U\cup K_{g}\right\rangle =H_{g}$. Therefore, $\left\langle k^{-1}Uk:k\in K_{g}\right\rangle $
is a normal subgroup of $H_{g}$ and because $U$ is not contained
in any proper normal subgroup of $H_{g}$, we have $\left\langle k^{-1}Uk:k\in K_{g}\right\rangle =H_{g}$.
This and (\ref{eq:erg1}) imply that $H_{g}\subset\gamma^{-1}H_{i}\gamma$.
Note that $\gamma\in SL_{d}\left(\mathbb{Z}\right)$ and by Remark
\ref{rational subgroups}, $H_{i}$ is defined over $\mathbb{Q}$.
Therefore, $\gamma^{-1}H_{i}\gamma$ is defined over $\mathbb{Q}$,
it follows from Theorem 7.7 of \cite{MR1278263} that $\overline{\gamma^{-1}H_{i}\gamma\cap SL_{d}\left(\mathbb{Q}\right)}=\gamma^{-1}H_{i}\gamma$.
Therefore Lemma 3.7 and Proposition 4.1 of \cite{2011arXiv1111.4428S}
imply that $\gamma^{-1}H_{i}\gamma=G_{g}$ which is a contradiction
and therefore $Y_{i}\left(\gamma\right)\subsetneq K_{g}$. This means
for all $1\leq i\leq k$, $Y_{i}\left(\gamma\right)$ is a submanifold
of strictly smaller dimension than $K_{g}$ and hence 
\begin{equation}
\nu_{g}\left(Y_{i}\left(\gamma\right)\right)=0.\label{eq:erg2.5}
\end{equation}
Note that because $C_{i}\subseteq X\left(H_{i},U\right)$, 
\[
K_{g}\cap\bigcup_{1\leq i\leq k}C_{i}\Gamma_{g}\subseteq K_{g}\cap\bigcup_{1\leq i\leq k}X\left(H_{i},U\right)\Gamma_{g}=\bigcup_{1\leq i\leq k}\bigcup_{\gamma\in\Gamma_{g}}Y_{i}\left(\gamma\right)
\]
and therefore \eqref{eq:erg2.5} implies 
\begin{equation}
\nu_{g}\left(K_{g}\cap\bigcup_{1\leq i\leq k}C_{i}\Gamma_{g}\right)=0.\label{eq:erg2}
\end{equation}
Let $\mathcal{D}$ be a compact subset of $G_{g}$ such that $K_{g}\subset\mathcal{D}$.
Then from (\ref{eq:erg2}), it follows that, for all $\eta>0$ there
exists a compact subset $F$ of $\mathcal{D}-\bigcup_{1\leq i\leq k}C_{i}\Gamma_{g}$,
such that 
\begin{equation}
\nu_{g}\left(F\cap K_{g}\right)\geq1-\eta.\label{eq:erg 3}
\end{equation}
From Theorem \ref{thm:dani and margulis}, for all $\epsilon>0$ there
exists a $T_{0}>0$, such that for all $x\in\left(F\cap K_{g}\right)/\Gamma_{g}$
and $T>T_{0}$,
\[
\left|\frac{1}{T}\int_{0}^{T}\phi\left(u_{t}x\right)dt-\int_{G_{g}/\Gamma_{g}}\phi d\mu_{g}\right|<\epsilon.
\]
Therefore if $k\in K_{g}$, $T>T_{0}$ and 
\[
\left|\frac{1}{T}\int_{0}^{T}\phi\left(u_{t}k\right)dt-\int_{G_{g}/\Gamma_{g}}\phi d\mu_{g}\right|>\epsilon,
\]
then $k\in K_{g}\setminus F$, but $\nu_{g}\left(K_{g}\setminus F\right)\leq\eta$
by (\ref{eq:erg 3}) and this implies (\ref{eq:erg conclusion 1}).\end{proof}
\begin{lem}
\label{cor:erg coro-1}Suppose $r_{1}\geq2$ and $r_{2}\geq1$. Let
$g\in\mathcal{C}_{SL}\left(r_{1},r_{2}\right)$ be arbitrary. Let
$U=\left\{ u_{t}:t\in\mathbb{R}\right\} $ be a one parameter unipotent
subgroup of $H_{g}$, not contained in any proper normal subgroup
of $H_{g}$. Let $\phi$ be a bounded continuous function on $G_{g}/\Gamma_{g}$.
Then for all $\epsilon>0$ and $\delta>0$ there exists a $T_{0}>0$
such that for all $T>T_{0}$, 
\[
\left|\frac{1}{\delta T}\int_{T}^{\left(1+\delta\right)T}\int_{K_{g}}\phi\left(u_{t}k\right)d\nu_{g}\left(k\right)dt-\int_{G_{g}/\Gamma_{g}}\phi d\mu_{g}\right|<\epsilon.
\]
\end{lem}
\begin{proof}
Let $\phi$ be a bounded continuous function on $G_{g}/\Gamma_{g}$.
Lemma \ref{thm:4.4} implies for all $\epsilon>0$, $\eta>0$ and
$d>0$ there exists a $T_{0}>0$ such that for all $T>T_{0}$, 
\begin{equation}
\nu_{g}\left(\left\{ k\in K_{g}:\left|\frac{1}{dT}\int_{0}^{dT}\phi\left(u_{t}k\right)dt-\int_{G_{g}/\Gamma_{g}}\phi d\mu_{g}\right|>\epsilon\right\} \right)\leq\eta.\label{eq:easy 2}
\end{equation}
Using (\ref{eq:easy 2}) with $d=1$ and $d=1+\delta$ we get that
for for all $\epsilon>0$ and $\eta>0$ there exists a subset $\mathcal{C}\subseteq K_{g}$
with $\nu_{g}\left(\mathcal{C}\right)\geq1-\eta$ such that for all
$k\in\mathcal{C}$ the following holds 
\[
\left|\int_{0}^{T}\phi\left(u_{t}k\right)dt-T\int_{G_{g}/\Gamma_{g}}\phi d\mu_{g}\right|<\epsilon T\quad\textrm{and}\quad\left|\int_{0}^{\left(1+\delta\right)T}\phi\left(u_{t}k\right)dt-\left(1+\delta\right)T\int_{G_{g}/\Gamma_{g}}\phi d\mu_{g}\right|<\left(1+\delta\right)T\epsilon.
\]
Hence for all $k\in\mathcal{C}$ we have
\begin{multline*}
\left|\int_{T}^{\left(1+\delta\right)T}\phi\left(u_{t}k\right)dt-\delta T\int_{G_{g}/\Gamma_{g}}\phi d\mu_{g}\right|\\
=\left|\int_{0}^{\left(1+\delta\right)T}\phi\left(u_{t}k\right)dt-\left(1+\delta\right)T\int_{G_{g}/\Gamma_{g}}\phi d\mu_{g}-\int_{0}^{T}\phi\left(u_{t}k\right)dt+T\int_{G_{g}/\Gamma_{g}}\phi d\mu_{g}\right|\\
\leq\left|\int_{0}^{T}\phi\left(u_{t}k\right)dt-T\int_{G_{g}/\Gamma_{g}}\phi d\mu_{g}\right|+\left|\int_{0}^{\left(1+\delta\right)T}\phi\left(u_{t}k\right)dt-\left(1+\delta\right)T\int_{G_{g}/\Gamma_{g}}\phi d\mu_{g}\right|\\
\leq\left(2+\delta\right)T\epsilon.
\end{multline*}
This means that for all $\delta>0$, $\eta>0$ and $\epsilon>0$,
\[
\nu_{g}\left(\left\{ k\in K_{g}:\left|\frac{1}{\delta T}\int_{T}^{\left(1+\delta\right)T}\phi\left(u_{t}k\right)dt-\int_{G_{g}/\Gamma_{g}}\phi d\mu_{g}\right|<\frac{\left(2+\delta\right)\epsilon}{\delta}\right\} \right)\geq1-\eta.
\]
Since we can make $\epsilon$ and $\eta$ as small as we wish this
implies the claim. \end{proof}
\begin{lem}
\label{cor:erg coro}Suppose $r_{1}\geq2$ and $r_{2}\geq1$. Let
$A=\left\{ a_{t}:t\in\mathbb{R}\right\} $ be a one parameter subgroup
of $H_{g}$, not contained in any proper normal subgroup of $H_{g}$,
such that there exists a continuous homomorphism $\rho:SL_{2}\left(\mathbb{R}\right)\rightarrow H_{g}$
with $\rho\left(D\right)=A$ and $\rho\left(SO\left(2\right)\right)\subset K_{g}$
where $D=\left\{ \left(\begin{smallmatrix}t & 0\\
0 & t^{-1}
\end{smallmatrix}\right):t>0\right\} $. Let $\phi$ be a continuous function on $G_{g}/\Gamma_{g}$ vanishing
outside of a compact set. Then for all $g\in\mathcal{C}_{\textrm{SL}}\left(r_{1},r_{2}\right)$
and $\epsilon>0$ there exists $T_{0}>0$ such that for all $t>T_{0}$,
\[
\left|\int_{K_{g}}\phi\left(a_{t}k\right)d\nu_{g}\left(k\right)-\int_{G_{g}/\Gamma_{g}}\phi d\mu_{g}\right|\leq\epsilon.
\]
\end{lem}
\begin{proof}
This is very similar to the proof of Theorem 4.4 from \cite{MR1609447}
and some details will be omitted. Fix $\epsilon>0$. Assume that $\phi$
is uniformly continuous. Let $u_{t}=\left(\begin{smallmatrix}1 & t\\
0 & 1
\end{smallmatrix}\right)$ and $w=\left(\begin{smallmatrix}0 & -1\\
1 & 0
\end{smallmatrix}\right)$. Then it is clear that $d_{t}=\left(\begin{smallmatrix}t & 0\\
0 & t^{-1}
\end{smallmatrix}\right)=b_{t}u_{t}k_{t}w$, where $b_{t}=\left(1+t^{-2}\right)^{-1/2}\left(\begin{smallmatrix}1 & 0\\
-t^{-1} & 1+t^{-2}
\end{smallmatrix}\right)$ and $k_{t}=\left(1+t^{-2}\right)^{-1/2}\left(\begin{smallmatrix}1 & t^{-1}\\
-t^{-1} & 1
\end{smallmatrix}\right).$ By our assumptions on $A$ there exists a continuous homomorphism
$\rho:SL_{2}\left(\mathbb{R}\right)\rightarrow H_{g}$ such that $\rho\left(D\right)=A$
and $\rho\left(SO\left(2\right)\right)\subset K_{g}$. Let $\rho\left(d_{t}\right)=d'_{t}$,
$\rho\left(b_{t}\right)=b'_{t}$, $\rho\left(k_{t}\right)=k'_{t}$
and $\rho\left(w\right)=w'$. Then for all $t>0$ and $g\in\mathcal{C}_{\textrm{SL}}\left(r_{1},r_{2}\right)$,
\begin{alignat}{1}
\int_{K_{g}}\phi\left(d'_{t}k\right)d\nu_{g}\left(k\right) & =\int_{K_{g}}\phi\left(b'_{t}u'_{t}k'_{t}w'k\right)d\nu_{g}\left(k\right)\nonumber \\
 & =\int_{K_{g}}\phi\left(b'_{t}u'_{t}k\right)d\nu_{g}\left(k\right),\label{eq:proof of translation to self adjoint 1}
\end{alignat}
since $k'_{t}$ and $w'\in K_{g}$. It follows from (\ref{eq:proof of translation to self adjoint 1})
that for $r,t>0$,
\begin{multline}
\left|\int_{K_{g}}\phi\left(d'_{t}k\right)d\nu_{g}\left(k\right)-\int_{K_{g}}\phi\left(u'_{rt}k\right)d\nu_{g}\left(k\right)\right|\\
\leq\left|\int_{K_{g}}\left(\phi\left(d'_{rt}k\right)-\phi\left(d'_{t}k\right)\right)d\nu_{g}\left(k\right)\right|+\left|\int_{K_{g}}\left(\phi\left(d'_{rt}k\right)-\phi\left(u'_{rt}k\right)\right)d\nu_{g}\left(k\right)\right|\\
=\left|\int_{K_{g}}\left(\phi\left(d'_{r}d'_{t}k\right)-\phi\left(d'_{t}k\right)\right)d\nu_{g}\left(k\right)\right|+\left|\int_{K_{g}}\left(\phi\left(b'_{rt}u'_{rt}k\right)-\phi\left(u'_{rt}k\right)\right)d\nu_{g}\left(k\right)\right|.\label{eq:trans 2}
\end{multline}
By uniform continuity, the fact that $\lim_{t\rightarrow\infty}b_{t}=I$
and (\ref{eq:trans 2}) imply there exists $T_{1}>0$ and $\delta>0$
such that for $t>T_{1}$ and $\left|r-1\right|<\delta$ we have 
\[
\left|\int_{K_{g}}\phi\left(d'_{t}k\right)d\nu_{g}\left(k\right)-\int_{K_{g}}\phi\left(u'_{rt}k\right)d\nu_{g}\left(k\right)\right|\leq\epsilon.
\]
Thus, if $T>T_{1}$ then 
\begin{equation}
\left|\int_{K_{g}}\phi\left(d'_{t}k\right)d\nu_{g}\left(k\right)-\frac{1}{\delta T}\int_{T}^{\left(1+\delta\right)T}\int_{K_{g}}\phi\left(u'_{t}k\right)d\nu_{g}\left(k\right)dt\right|\leq\epsilon.\label{eq:translation 4}
\end{equation}
Combining (\ref{eq:translation 4}) with Lemma \ref{cor:erg coro}
via the triangle inequality finishes the proof of the Lemma. 
\end{proof}
The section is completed by the proof of the main ergodic result whose
proof follows that of Theorem 3.5 in \cite{MR1609447}. 
\begin{proof}[Proof of Theorem \ref{thm:ergodic result-1}]
Assume that $\phi$ is non-negative. Let $A\left(r\right)=\left\{ x\in G_{g}/\Gamma_{g}:\alpha\left(x\right)>r\right\} $.
Choose a continuous non-negative function $g_{r}$ on $G_{g}/\Gamma_{g}$
such that $g_{r}\left(x\right)=1$ if $x\in A\left(r+1\right)$, $g_{r}\left(x\right)=0$
if $x\notin A\left(r\right)$ and $0\leq g_{r}\left(x\right)\leq1$
if $x\in A\left(r\right)\setminus A\left(r+1\right)$. Then 
\begin{equation}
\int_{K_{g}}\phi\left(a_{t}k\right)d\nu_{g}\left(k\right)=\int_{K_{g}}\phi\left(a_{t}k\right)g_{r}\left(a_{t}k\right)d\nu_{g}\left(k\right)+\int_{K_{g}}\left(\phi\left(a_{t}k\right)-\phi\left(a_{t}k\right)g_{r}\left(a_{t}k\right)\right)d\nu_{g}\left(k\right).\label{eq:end of sec 1 0.5}
\end{equation}
Let $\beta=2-\delta$ then for $x\in G_{g}/\Gamma_{g}$, 
\begin{alignat*}{1}
\phi\left(x\right)g_{r}\left(x\right) & \leq C\alpha\left(x\right)^{2-\beta}g_{r}\left(x\right)\\
 & =C\alpha\left(x\right)^{2-\beta/2}g_{r}\left(x\right)\alpha\left(x\right)^{-\beta/2}\leq Cr^{-\beta/2}\alpha\left(x\right)^{2-\beta/2}.
\end{alignat*}
The last inequality is true because $g_{r}\left(x\right)=0$ if $\alpha\left(x\right)\leq r$.
Therefore 
\begin{equation}
\int_{K_{g}}\phi\left(a_{t}k\right)g_{r}\left(a_{t}k\right)d\nu_{g}\left(k\right)\leq Cr^{-\beta/2}\int_{K_{g}}\alpha\left(a_{t}k\right)^{2-\beta/2}d\nu_{g}\left(k\right).\label{eq:end of sec 1 1}
\end{equation}
Since $g\in\mathcal{C}_{\textrm{SL}}\left(r_{1},r_{2}\right)$, $r_{1}\geq3$
and $r_{2}\geq1$ Theorem \ref{thm:upper bound-1} part I implies
there exists $B$ such that 
\begin{alignat*}{2}
\int_{K_{g}}\alpha\left(a_{t}k\right)^{2-\beta/2}d\nu_{g}\left(k\right) & =\int_{K_{I}}\alpha\circ g^{-1}\left(a_{t}kg\right)^{2-\beta/2}d\nu_{I}\left(k\right) & \leq c\left(g\right)\int_{K_{I}}\alpha\left(a_{t}kg\right)^{2-\beta/2}d\nu_{I}\left(k\right)<B
\end{alignat*}
for all $t\geq0$. Then (\ref{eq:end of sec 1 1}) implies that 
\begin{equation}
\int_{K_{g}}\phi\left(a_{t}k\right)g_{r}\left(a_{t}k\right)d\nu_{g}\left(k\right)\leq BCr^{-\beta/2}.\label{eq:end of sec 1 2}
\end{equation}
For all $\epsilon>0$ there exists a compact subset, $\mathcal{C}$
of $G_{g}/\Gamma_{g}$ such that $\mu_{g}\left(\mathcal{C}\right)\geq1-\epsilon$.
The function $\alpha$ is bounded on $\mathcal{C}$ and hence for
all $\epsilon>0$, 
\[
\lim_{r\rightarrow\infty}\mu_{g}\left(A\left(r\right)\right)=\lim_{r\rightarrow\infty}\left(\mu_{g}\left(\left\{ x\in\mathcal{C}:\alpha\left(x\right)>r\right\} \right)+\mu_{g}\left(\left\{ x\in\left(G_{g}/\Gamma_{g}\right)\setminus\mathcal{C}:\alpha\left(x\right)>r\right\} \right)\right)\leq\epsilon.
\]
This means that 
\begin{equation}
\lim_{r\rightarrow\infty}\mu_{g}\left(A\left(r\right)\right)=0.\label{eq:limit of sets tends to zero}
\end{equation}
Note that 
\begin{equation}
\int_{G_{g}/\Gamma_{g}}\phi\left(x\right)g_{r}\left(x\right)d\mu_{g}\left(x\right)\leq\int_{A\left(r\right)}\phi\left(x\right)d\mu_{g}\left(x\right).\label{eq:bound on funny integral}
\end{equation}
Since $\phi\in L^{1}\left(G_{g}/\Gamma_{g}\right)$, \prettyref{eq:limit of sets tends to zero}
and \prettyref{eq:bound on funny integral} imply that 
\begin{equation}
\lim_{r\rightarrow\infty}\int_{G_{g}/\Gamma_{g}}\phi\left(x\right)g_{r}\left(x\right)d\mu_{g}\left(x\right)=0.\label{eq:limit of integral}
\end{equation}
Since the function $\phi\left(x\right)-\phi\left(x\right)g_{r}\left(x\right)$
is continuous and has compact support, Lemma \ref{cor:erg coro} implies
for all $\epsilon>0$ and $g\in\mathcal{C}_{\textrm{SL}}\left(r_{1},r_{2}\right)$
there exists $T_{0}>0$ such that for all $t>T_{0}$, 
\begin{equation}
\left|\int_{K_{g}}\left(\phi\left(a_{t}k\right)-\phi\left(a_{t}k\right)g_{r}\left(a_{t}k\right)\right)d\nu_{g}\left(k\right)-\int_{G_{g}/\Gamma_{g}}\left(\phi\left(x\right)-\phi\left(x\right)g_{r}\left(x\right)\right)d\mu_{g}\left(x\right)\right|<\frac{\epsilon}{2}.\label{eq:end of sec 1 3}
\end{equation}
It is straight forward to check that (\ref{eq:end of sec 1 0.5}),
(\ref{eq:end of sec 1 2}), \prettyref{eq:limit of integral} and
(\ref{eq:end of sec 1 3}) imply the conclusion of the Theorem if
$r$ is sufficiently large. 
\end{proof}

\section{Proof of Theorem \ref{main theorem-1}.\label{sec:Proof-of-Theorem}}

The proof of Theorem \ref{main theorem-1} follows the same route
as that of Sections 3.4-3.5 of \cite{MR1609447}. The main modification
we make in order to handle the present situation is that we work inside
the surface $X_{g}\left(\mathbb{R}\right)$ rather than in the whole
of $\mathbb{R}^{d}$. For $t\in\mathbb{R}$ and $v\in\mathbb{R}^{d}$
define a linear map $a_{t}$ by 
\[
a_{t}v=\left(v_{1},\ldots,v_{s},e^{-t}v_{s+1},v_{s+2},\ldots,e^{t}v_{d}\right).
\]
Note that the one parameter group $\left\{ \hat{a}_{t}:t\mathbb{\in R}\right\} =g^{-1}\left\{ a_{t}:t\in\mathbb{R}\right\} g\subset H_{g}$
and that there exists a continuous homomorphism $\rho:SL_{2}\left(\mathbb{R}\right)\rightarrow H_{g}$
with $\rho\left(D\right)=\left\{ \hat{a}_{t}:t\mathbb{\in R}\right\} $
and $\rho\left(SO\left(2\right)\right)\subset K_{g}$ where $D=\left\{ \left(\begin{smallmatrix}t & 0\\
0 & t^{-1}
\end{smallmatrix}\right):t>0\right\} $. Moreover note that $\left\{ a_{t}:t\in\mathbb{R}\right\} $ is self
adjoint, not contained in any normal subgroup of $H_{g}$ and the
only eigenvalues of $a_{t}$ are $e^{-t},1$ and $e^{t}$. In other
words, $\left\{ \hat{a}_{t}:t\mathbb{\in R}\right\} $ satisfies the
conditions of Theorem \ref{thm:ergodic result-1} and Theorem \ref{thm:upper bound-1}.
For any natural number $n$, let $S^{n-1}$ denote the unit sphere
in a $n$ dimensional Euclidean space and let $\gamma_{n}=\textrm{Vol}\left(S^{n}\right)$
and $c_{r_{1},r_{2}}=\gamma_{r_{1}-1}\gamma_{r_{2}-1}$ then define
\begin{equation}
C_{1}=c_{r_{1},r_{2}}2^{\left(2-r_{1}-r_{2}\right)/2}=c_{r_{1},r_{2}}2^{\left(2-d+s\right)/2}.\label{eq:def of C_1}
\end{equation}

\subsection{Proof of Theorem \ref{prop:crude upper bound}.}

In Lemma \ref{lemma about evctors} it is shown that it is possible
to approximate certain integrals over $K_{g}$ by integrals over $\mathbb{R}^{d-s-2}$.
The integral over $\mathbb{R}^{d-s-2}$ can be used like the characteristic
function of $R\times A\left(T/2,T\right)$, in particular Theorem
\ref{prop:crude upper bound} is proved as an application of Lemma
\ref{lemma about evctors}. It should be noted that Lemma \ref{lemma about evctors}
is analogous to Lemma 3.6\texttt{ }from \cite{MR1609447} and its
proof is similar\texttt{. }
\begin{lem}
\label{lemma about evctors}Let $f$ be a continuous function of compact
support on $\mathbb{R}_{+}^{d}=\left\{ v\in\mathbb{R}^{d}:\left\langle v,e_{s+1}\right\rangle >0\right\} $
and for $g\in\mathcal{C}_{\textrm{SL}}\left(r_{1},r_{2}\right)$ let
\[
J_{f,g}\left(\ell_{1},\ldots,\ell_{s},r\right)=\frac{1}{r^{d-s-2}}\int_{\mathbb{R}^{d-s-2}}f\left(\ell_{1},\ldots,\ell_{s},r,v_{s+2},\ldots,v_{d-1},v_{d}\right)dv_{s+2}\ldots dv_{d-1},
\]
where $v_{d}=\left(a-Q_{0}^{g}\left(\ell_{1},\ldots,\ell_{s},0,v_{s+2},\ldots,v_{d-1},0\right)\right)/2r,$
so that $Q_{0}^{g}\left(\ell_{1},\ldots,\ell_{s},r,v_{s+2},\ldots,v_{d-1},v_{d}\right)=a$.
Then for every $\epsilon>0$ there exists $T_{0}>0$ such that for
every $t$ with $e^{t}>T_{0}$ and every $v\in\mathbb{R}_{+}^{d}$
with $\left\Vert v\right\Vert >T_{0}$,
\[
\left|C_{1}e^{\left(d-s-2\right)t}\int_{K_{g}}f\left(\hat{a}_{t}kv\right)d\nu_{g}\left(k\right)-J_{f,g}\left(M_{0}^{g}\left(v\right),\left\Vert v\right\Vert e^{-t}\right)\right|<\epsilon.
\]
\end{lem}
\begin{proof}
By Lemma 2.2 of \cite{2011arXiv1111.4428S}, for all $g\in\mathcal{C}_{\textrm{SL}}\left(r_{1},r_{2}\right)$
there exists a basis of $\mathbb{R}^{d}$, denoted by $b_{1},\ldots,b_{d}$
such that 
\[
Q_{0}^{g}\left(v\right)=Q_{1,\ldots,s}\left(v\right)+2v_{s+1}v_{d}+\sum_{i=s+2}^{s+r_{1}}v_{i}^{2}-\sum_{i=s+r_{1}+1}^{d-1}v_{i}^{2}\quad\textrm{,}\quad M_{0}^{g}\left(v\right)=\left(v_{1},\ldots,v_{s}\right),
\]
and
\[
\hat{a}_{t}\left(v\right)=\left(v_{1},\ldots v_{s},e^{-t}v_{s+1},v_{s+2},\ldots v_{d-1},e^{t}v_{d}\right),
\]
where $v_{i}=\left\langle v,b_{i}\right\rangle $ for $1\leq i\leq d$
and $Q_{1,\ldots,s}\left(v\right)$ is a non degenerate quadratic
form in variables $v_{1},\ldots,v_{s}$. Let $E$ denote the support
of $f$. Let $c_{1}=\inf_{v\in E}\left\langle v,b_{s+1}\right\rangle $,
$c_{2}=\sup_{v\in E}\left\langle v,b_{s+1}\right\rangle $. From the
definition of $\hat{a}_{t}$ it follows that $f\left(\hat{a}_{t}w\right)=0$
unless
\begin{gather}
\left|\left\langle w,b_{s+1}\right\rangle \left\langle w,b_{d}\right\rangle \right|\leq\beta,\label{eq:1}\\
c_{1}\leq\left\langle w,b_{s+1}\right\rangle e^{-t}\leq c_{2},\label{eq:2}\\
\pi'\left(w\right)\in\pi'\left(E\right),\label{eq:3}
\end{gather}
where $\beta$ depends only on $E$ and $\pi'$ denotes the projection
onto the span of $b_{1},\ldots,b_{s},b_{s+2},\ldots,b_{d-1}$. For
$w$ satisfying (\ref{eq:1}) and (\ref{eq:2}) we have $\left\langle w,b_{d}\right\rangle =O\left(e^{-t}\right)$.
This, together with (\ref{eq:3}) and (\ref{eq:2}) imply that if
$f\left(\hat{a}_{t}w\right)\neq0$ and $t$ is large, then 
\begin{equation}
\left\Vert w\right\Vert =\left\langle w,b_{s+1}\right\rangle +O\left(e^{-t}\right).\label{eq:4}
\end{equation}
Note that by (\ref{eq:4}), 
\begin{equation}
\left\langle \hat{a}_{t}w,b_{s+1}\right\rangle =\left\langle w,b_{s+1}\right\rangle e^{-t}=e^{-t}\left\Vert w\right\Vert +O\left(e^{-2t}\right),\label{eq:5}
\end{equation}
and 
\begin{equation}
\left\langle \hat{a}_{t}w,b_{i}\right\rangle =\left\langle w,b_{i}\right\rangle ,\quad\quad\textrm{for }1\leq i\leq s,\textrm{ or }s+2\leq i\leq d-1.\label{eq:6}
\end{equation}
Finally, 
\begin{alignat}{1}
\left\langle \hat{a}_{t}w,b_{d}\right\rangle  & =\left(Q_{0}^{g}\left(w\right)-Q_{0}^{g}\left(\left\langle w,b_{1}\right\rangle ,\ldots,\left\langle w,b_{s}\right\rangle ,0,\left\langle w,b_{s+1}\right\rangle ,\ldots,\left\langle w,b_{d-1}\right\rangle ,0\right)\right)/2\left\langle \hat{a}_{t}w,b_{s+1}\right\rangle \label{eq:7}\\
 & =\left(Q_{0}^{g}\left(w\right)-Q_{0}^{g}\left(\left\langle w,b_{1}\right\rangle ,\ldots,\left\langle w,b_{s}\right\rangle ,0,\left\langle w,b_{s+1}\right\rangle ,\ldots,\left\langle w,b_{d-1}\right\rangle ,0\right)\right)/2e^{-t}\left\Vert w\right\Vert +O\left(e^{-t}\right).\nonumber 
\end{alignat}
Hence, using (\ref{eq:5}), (\ref{eq:6}) and (\ref{eq:7}) together
with the uniform continuity of $f$, applied with $w=kv$ for $v\in\mathbb{R}_{+}^{d}$
and $k\in K_{g}$, we see that for all $\delta>0$ there exists a
$t_{0}>0$ so that if $t>t_{0}$ then 
\begin{equation}
\left|f\left(\hat{a}_{t}kv\right)-f\left(v_{1},\ldots,v_{s},\left\Vert v\right\Vert e^{-t},\left\langle kv,b_{s+1}\right\rangle ,\ldots,\left\langle kv,b_{d-1}\right\rangle ,v_{d}\right)\right|<\delta,\label{eq:8}
\end{equation}
where $v_{d}$ is determined by 
\[
Q_{0}^{g}\left(v_{1},\ldots,v_{s},\left\Vert v\right\Vert e^{-t},\left\langle kv,b_{s+1}\right\rangle ,\ldots,\left\langle kv,b_{d-1}\right\rangle ,v_{d}\right)=Q_{0}^{g}\left(v\right)=a.
\]
Change basis by letting $f_{s+1}=\left(b_{s+1}+b_{d}\right)/\sqrt{2}$,
$f_{d}=\left(b_{s+1}-b_{d}\right)/\sqrt{2}$ and $f_{i}=b_{i}$ for
\foreignlanguage{english}{$1\leq i\leq s,\textrm{ or }s+2\leq i\leq d-1.$}
In this basis $K_{g}\cong SO\left(r_{1}\right)\times SO\left(r_{2}\right)$
consists of orthogonal matrices preserving the subspaces $L_{1}=\left\langle f_{1},\ldots,f_{s}\right\rangle ,$
$L_{2}=\left\langle f_{s+1},\ldots,f_{s+r_{1}}\right\rangle $ and
$L_{3}=\left\langle f_{s+r_{1}+1},\ldots,f_{d}\right\rangle $. For
$i=1,2$ or $3$, let $\pi_{i}$ denote the orthogonal projection
onto $L_{i}$. Write $\rho_{i}=\left\Vert \pi_{i}\left(v\right)\right\Vert $;
then the orbit $K_{g}v$ is product of a point and two spheres $\left\{ v_{1},\ldots,v_{s}\right\} \times\rho_{2}S^{r_{1}-1}\times\rho_{3}S^{r_{2}-1}$,
where $S^{r_{1}-1}$ denotes the unit sphere in $L_{2}$ and $S^{r_{2}-1}$
the unit sphere in $L_{3}$. 

Suppose $w\in K_{g}v$ is such that $f\left(\hat{a}_{t}w\right)\neq0$.
Then from (\ref{eq:1}) and (\ref{eq:2}) it follows that $\left\langle w,b_{d}\right\rangle =O\left(e^{-t}\right)$.
Now, set $w_{i}=\left\langle w,f_{i}\right\rangle $, then $w_{s+1}=2^{-1/2}\left\langle w,b_{s+1}\right\rangle +O\left(e^{-t}\right)$,
$w_{d}=2^{-1/2}\left\langle w,b_{s+1}\right\rangle +O\left(e^{-t}\right)$
and for \foreignlanguage{english}{$1\leq i\leq s,\textrm{ or }s+2\leq i\leq d-1$,}
\foreignlanguage{english}{$w_{i}=O\left(1\right)$}. Hence for $i=2$
or $3$,
\begin{equation}
\rho_{i}=\left\Vert \pi_{i}\left(w\right)\right\Vert =2^{-1/2}\left\langle w,b_{s+1}\right\rangle +O\left(e^{-t}\right)=2^{-1/2}\left\Vert w\right\Vert +O\left(e^{-t}\right),\label{eq:0.9}
\end{equation}
where the last estimate follows from (\ref{eq:4}). 

By integrating (\ref{eq:8}) with respect to $K_{g}$ we see that
for all $\epsilon>0$ there exists a $t_{0}>0$ so that if $t>t_{0}$
then
\begin{equation}
\left|\int_{K_{g}}f\left(\hat{a}_{t}kv\right)d\nu_{g}\left(k\right)-\int_{K_{g}}f\left(v_{1},\ldots,v_{s},\left\Vert v\right\Vert e^{-t},\left\langle kv,b_{s+1}\right\rangle ,\ldots,\left\langle kv,b_{d-1}\right\rangle ,v_{d}\right)d\nu_{g}\left(k\right)\right|<\epsilon.\label{eq:10}
\end{equation}
\textcolor{black}{Equation (\ref{eq:3}) implies that if $f\left(\hat{a}_{t}kv\right)\neq0$,
then $kv$ is within a bounded distance from $\rho_{2}f_{s+1}+\rho_{3}f_{d}$.
As $\left\Vert v\right\Vert $ increases }so do the $\rho_{i}$ and
the normalised Haar measure on $\rho_{2}S^{r_{1}-1}$ near $\rho_{2}f_{s+1}$
tends to $\left(1/\textrm{Vol}\left(\rho_{2}S^{r_{1}-1}\right)\right)dv_{s+2}\ldots dv_{s+r_{1}}$
and similarly the Haar measure on $\rho_{3}S^{r_{2}-1}$ near $\rho_{3}f_{d}$
tends to $\left(1/\textrm{Vol}\left(\rho_{3}S^{r_{2}-1}\right)\right)dv_{s+r_{1}+1}\ldots dv_{d-1}$.
This means that as $\left\Vert v\right\Vert $ tends to infinity the
second integral in (\ref{eq:10}) tends to \foreignlanguage{english}{
\begin{multline}
\frac{\rho_{2}^{1-r_{1}}\rho_{3}^{1-r_{2}}}{c_{r_{1},r_{2}}}\int_{\mathbb{R}^{d-s-2}}f\left(v_{1},\ldots,v_{s},\left\Vert v\right\Vert e^{-t},v_{s+1},\ldots,v_{d}\right)dv_{s+2}\ldots dv_{d-1}\\
=\frac{\left(\left\Vert v\right\Vert e^{-t}\right)^{d-s-2}}{\rho_{2}^{r_{1}-1}\rho_{3}^{r_{2}-1}c_{r_{1},r_{2}}}J_{f,g}\left(M_{0}^{g}\left(v\right),\left\Vert v\right\Vert e^{-t}\right).\label{eq:14}
\end{multline}
}Because (\ref{eq:0.9}) implies that $\rho_{2}^{r_{1}-1}\rho_{3}^{r_{2}-1}=2^{\left(s+2-d\right)/2}\left\Vert v\right\Vert ^{d-s-2}+O(e^{-t})$
we can use (\ref{eq:10}) and (\ref{eq:14}) to get that for all $\epsilon>0$
there exists a $t_{0}>0$ so that if $t>t_{0}$ and $\left\Vert v\right\Vert >t_{0}$
then\foreignlanguage{english}{ }
\[
\left|\int_{K_{g}}f\left(\hat{a}_{t}kv\right)d\nu_{g}\left(k\right)-\frac{e^{t\left(s+2-d\right)}}{C_{1}}J_{f,g}\left(M_{0}^{g}\left(v\right),\left\Vert v\right\Vert e^{-t}\right)\right|<\epsilon.
\]
By dividing through by the factor $\frac{e^{t\left(s+2-d\right)}}{C_{1}}$
we obtain the desired conclusion. 
\end{proof}
For $f_{1}$ and $f_{2}$ continuous functions of compact support
on $\mathbb{R}_{+}^{d}=\left\{ v\in\mathbb{R}^{d}:\left\langle v,e_{s+1}\right\rangle >0\right\} $,
define $J_{f_{1},g}+J_{f_{2},g}=J_{f_{1}+f_{2},g}$ and $J_{f_{1},g}J_{f_{2},g}=J_{f_{1}f_{2},g}$.
These operations make the collection of functions of the form $J_{f,g}$
into an algebra of real valued functions on the set $\mathbb{R}^{s}\times\left\{ v\in\mathbb{R}:v>0\right\} $.
Denote by this algebra by $\mathcal{A}$. The following Lemma will
be used in the proofs of Theorem \ref{prop:crude upper bound} and
Theorem \ref{main theorem-1}. 
\begin{lem}
\label{lem:stone weirstarss}$\mathcal{A}$ is dense in $C_{c}\left(\mathbb{R}^{s}\times\left\{ v\in\mathbb{R}:v>0\right\} \right)$. \end{lem}
\begin{proof}
Let $B$ be a compact subset of $\mathbb{R}^{s}\times\left\{ v\in\mathbb{R}:v>0\right\} $.
Let $\mathcal{A}_{B}$ denote the subalgebra of $\mathcal{A}$ of
functions with support $B$. It is straightforward to check that the
algebra $\mathcal{A}_{B}$ separates points in $B$ and does not vanish
at any point in $B$. Therefore, by the Stone-Weierstrass Theorem
(cf. \cite{MR0385023}, Theorem 7.32) $\mathcal{A}_{B}$ is dense
in the space of continuous functions on $B$. Since $B$ is arbitrary
this implies the claim. 
\end{proof}

\begin{proof}[Proof of Theorem \ref{prop:crude upper bound}.]
Let $\epsilon>0$ be arbitrary and $g\in\mathcal{C}_{\textrm{SL}}\left(r_{1},r_{2}\right)$.
By Lemma \ref{lem:stone weirstarss} there exists a continuous non-negative
function $f$ on $\mathbb{R}_{+}^{d}$ of compact support so that
$J_{f,g}\geq1+\epsilon$ on $R\times\left[1,2\right]$. Then if $v\in\mathbb{R}^{d}$
satisfies $e^{t}\leq\Vert v\Vert\leq2e^{t}$, $M_{0}^{g}\left(v\right)\in R$
and $Q_{0}^{g}\left(v\right)=a$ then $J_{f,g}\left(M_{0}^{g}\left(v\right),\left\Vert v\right\Vert e^{-t}\right)\geq1+\epsilon$.
Then by Lemma \ref{lemma about evctors}, for sufficiently large $t$,
\[
C_{1}e^{\left(d-s-2\right)t}\int_{K_{g}}f\left(\hat{a}_{t}kv\right)d\nu_{g}\left(k\right)\geq1
\]
if $e^{t}\leq\Vert v\Vert\leq2e^{t}$, $M_{0}^{g}\left(v\right)\in R$,
and $Q_{0}^{g}\left(v\right)=a$. Then summing over $v\in X_{g}\left(\mathbb{Z}\right)$,
we get 
\begin{align}
\left|X_{g}\left(\mathbb{Z}\right)\cap V_{M}\left(\left[a,b\right]\right)\cap A\left(e^{t},2e^{t}\right)\right| & \leq\sum_{v\in X_{g}\left(\mathbb{Z}\right)}C_{1}e^{\left(d-s-2\right)t}\int_{K_{g}}f\left(\hat{a}_{t}kv\right)d\nu_{g}\left(k\right)\nonumber \\
 & =C_{1}e^{\left(d-s-2\right)t}\int_{K_{g}}F_{f,g}\left(\hat{a}_{t}k\right)d\nu_{g}\left(k\right).\label{eq:prop 1}
\end{align}
Note that 
\begin{equation}
\int_{K_{g}}F_{f,g}\left(\hat{a}_{t}k\right)d\nu_{g}\left(k\right)=\int_{K_{I}}F_{f,g}\left(g^{-1}a_{t}kg\right)d\nu_{I}\left(k\right).\label{eq:prop 2}
\end{equation}
By (\ref{eq:upper bound for function}) we have the bound $F_{f,g}\left(x\right)\leq c\left(f\right)\alpha\left(x\right)$
for all $x\in G_{g}/\Gamma_{g}$ where $c\left(f\right)$ is a constant
depending only on $f$. Since $g\in\mathcal{C}_{\textrm{SL}}\left(r_{1},r_{2}\right)$,
part I of Theorem \ref{thm:upper bound-1} implies that if $r_{1}\geq3$
and $r_{2}\geq1$ then 
\begin{equation}
\int_{K_{I}}F_{f,g}\left(g^{-1}a_{t}kg\right)d\nu_{I}\left(k\right)<c\left(f\circ g^{-1}\right)\int_{K_{I}}\alpha\left(a_{t}kg\right)d\nu_{I}\left(k\right)<\infty.\label{eq:prop 3}
\end{equation}
In the case when $r_{1}=2$ and $r_{2}=1$ or $r_{1}=r_{2}=2$ part
II of Theorem \ref{thm:upper bound-1} implies that for all $g\in\mathcal{C}_{\textrm{SL}}\left(r_{1},r_{2}\right)$
there exists a constant $C$ so that 
\begin{equation}
\int_{K_{I}}F_{f,g}\left(g^{-1}a_{t}kg\right)d\nu_{I}\left(k\right)<c\left(f\circ g^{-1}\right)\int_{K_{I}}\alpha\left(a_{t}kg\right)d\nu_{I}\left(k\right)<Ct.\label{eq:prop 3-1}
\end{equation}
Hence, (\ref{eq:prop 1}), (\ref{eq:prop 2}) and (\ref{eq:prop 3})
imply that as long as $r_{1}\geq3$ and $r_{2}\geq1$ there exists
a constant $C_{2}$ such that 
\[
\left|X_{g}\left(\mathbb{Z}\right)\cap V_{M}\left(R\right)\cap A\left(e^{t},2e^{t}\right)\right|\leq C_{2}e^{\left(d-s-2\right)t}.
\]
Similarly, (\ref{eq:prop 1}), (\ref{eq:prop 2}) and (\ref{eq:prop 3-1})
imply that if $r_{1}=2$ and $r_{2}=1$ or $r_{1}=r_{2}=2$, then
\[
\left|X_{g}\left(\mathbb{Z}\right)\cap V_{M}\left(R\right)\cap A\left(e^{t},2e^{t}\right)\right|\leq C_{2}te^{\left(d-s-2\right)t}.
\]
Since we can write $T=e^{t}$ and 
\[
A\left(0,T\right)=\lim_{n\rightarrow\infty}\left(A\left(0,T/2^{n}\right)\bigcup_{i=1}^{n}A\left(T/2^{i},T/2^{i-1}\right)\right),
\]
the Theorem follows by summing a geometric series.
\end{proof}
Theorem \ref{prop:crude upper bound} has the following Corollary
which is comparable with Proposition 3.7 from \cite{MR1609447} and
will be used in the proof of Theorem \ref{main theorem-1}.
\begin{cor}
\label{lem:sum approx}Let $f$ be a continuous function of compact
support on $\mathbb{R}_{+}^{d}$. Then for every $\epsilon>0$ and
$g\in\mathcal{C}_{\textrm{SL}}\left(r_{1},r_{2}\right)$ there exists
$t_{0}>0$ so that for $t>t_{0}$, 
\begin{equation}
\left|e^{-\left(d-s-2\right)t}\sum_{v\in X_{g}\left(\mathbb{Z}\right)}J_{f,g}\left(M_{0}^{g}\left(v\right),\left\Vert v\right\Vert e^{-t}\right)-C_{1}\int_{K_{g}}F_{f,g}\left(\hat{a}_{t}k\right)d\nu_{g}\left(k\right)\right|<\epsilon.\label{eq:3.3}
\end{equation}
\end{cor}
\begin{proof}
Since $J_{f,g}$ has compact support, the number of non zero terms
in the sum on the left hand side of (\ref{eq:3.3}) is bounded by
$ce^{\left(d-s-2\right)t}$ because of Theorem \ref{prop:crude upper bound}.
Hence summing the result of Lemma \ref{lemma about evctors} over
$v\in X_{g}\left(\mathbb{Z}\right)$ proves (\ref{eq:3.3}). 
\end{proof}

\subsection{Volume estimates.}

For a compactly supported function $h$ on $\mathbb{R}^{s}\times\mathbb{R}^{d}\setminus\left\{ 0\right\} $
we define 
\[
\Theta\left(h,T\right)=\int_{X_{g}\left(\mathbb{R}\right)}h\left(M_{0}^{g}\left(v\right),v/T\right)dm_{g}\left(v\right).
\]
For $\mathcal{X}\subseteq\mathbb{R}^{d}$ we will use the notation
$\textrm{Vol}_{X_{g}}\left(\mathcal{X}\right)=\int_{X_{g}\left(\mathbb{R}\right)}\mathbb{1}_{\mathcal{X}\cap X_{g}\left(\mathbb{R}\right)}dm_{g}$
to mean the volume of $\mathcal{X}$ with respect to the volume measure
on $X_{g}\left(\mathbb{R}\right)$. 

The following Lemma and its Corollary are analogous to Lemma 3.8 from
\cite{MR1609447} and the proofs share some similarities, although
it is here that the fact we are integrating over $X_{g}\left(\mathbb{R}\right)$
rather than the whole of $\mathbb{R}^{d}$ becomes an important distinction.
In Lemma \ref{lem:further approximations} we compute $\lim_{T\rightarrow\infty}\frac{1}{T^{d-s-2}}\Theta\left(h,T\right)$,
here it is crucial that $h$ is not defined on $\mathbb{R}^{s}\times\left\{ 0\right\} $;
if it was, using the fact that $h$ can be bounded by an integrable
function, one could directly pass the limit inside the integral and
the limit would be 0. The basic strategy of Lemma \ref{lem:further approximations}
is that we evaluate the integral \foreignlanguage{english}{$\int_{X_{g}\left(\mathbb{R}\right)}dm_{g}$}
by switching to polar coordinates. This has the effect that the integral
changes into an integral over two spheres, then we approximate the
spheres by an orbit of $K_{g}$ and an integral over the coordinates
fixed by $K_{g}$. 
\begin{lem}
\label{lem:further approximations}Suppose that $h$ is a continuous
function of compact support in $\mathbb{R}^{s}\times\mathbb{R}^{d}\setminus\left\{ 0\right\} $.
Then 
\[
\lim_{T\rightarrow\infty}\frac{1}{T^{d-s-2}}\Theta\left(h,T\right)=C_{1}\int_{K_{g}}\int_{0}^{\infty}\int_{\mathbb{R}^{s}}h\left(z,rke_{0}\right)r^{d-s-2}dz\frac{dr}{2r}d\nu_{g}\left(k\right),
\]
where $e_{0}$ is a unit vector in $\mathbb{R}^{d}$ and $C_{1}$
is the constant defined by \prettyref{eq:def of C_1}.\end{lem}
\begin{proof}
By Lemma 2.2 of \cite{2011arXiv1111.4428S}, for all $g\in\mathcal{C}_{\textrm{SL}}\left(r_{1},r_{2}\right)$
there exists a basis of $\mathbb{R}^{d}$, denoted by $f_{1},\ldots,f_{d}$
such that 
\[
Q_{0}^{g}\left(v\right)=\sum_{i=1}^{s_{1}}v_{i}^{2}-\sum_{i=s_{1}+1}^{s}v_{i}^{2}+\sum_{i=s+1}^{s+r_{1}}v_{i}^{2}-\sum_{i=s+r_{1}+1}^{d}v_{i}^{2}\quad\textrm{and}\quad M_{0}^{g}\left(v\right)=J\left(v_{1},\ldots,v_{s}\right),
\]
where $v_{i}=\left\langle v,f_{i}\right\rangle $ for $1\leq i\leq d$
, $J\in GL_{s}\left(\mathbb{R}\right)$, $s_{1}$ is a non-negative
integer such that $r_{1}+s_{1}=p$ and $s_{2}$ is a non-negative
integer such that $r_{2}+s_{2}=d-p$. Let $L_{1}=\left\langle v_{1},\dots,v_{s_{1}},v_{s+1},\dots,v_{s+r_{1}}\right\rangle $,
$L_{2}=\left\langle v_{s_{1}+1},\dots,v_{s},v_{s+r_{1}+1},\dots,v_{d}\right\rangle $,
$S^{p-1}$ be the unit sphere inside $L_{1}$ and $S^{d-p-1}$ be
the unit sphere inside $L_{2}$. Let $\alpha\in S^{p-1}$ and $\beta\in S^{d-p-1}$.
Using polar coordinates, we can parametrise  $v\in X_{g}\left(\mathbb{R}\right)$
so that 
\begin{equation}
v_{i}=\begin{cases}
\sqrt{a}\alpha_{i}\cosh t & \textrm{for }1\leq i\leq s_{1}\\
\sqrt{a}\beta_{i-s_{1}}\sinh t & \textrm{for }s_{1}+1\leq i\leq s\\
\sqrt{a}\alpha_{i-s+s_{1}}\cosh t & \textrm{for }s+1\leq i\leq s+r_{1}\\
\sqrt{a}\beta_{i-s_{1}-r_{1}}\sinh t & \textrm{for }s+r_{1}+1\leq i\leq d.
\end{cases}\label{eq:5.40}
\end{equation}
In these coordinates we may write 
\[
dm_{g}\left(v\right)=\frac{a^{\left(d-2\right)/2}}{2}\cosh^{p-1}t\sinh^{q-1}tdtd\xi\left(\alpha,\beta\right)=P\left(e^{t}\right)dtd\xi\left(\alpha,\beta\right),
\]
where $P\left(x\right)=\frac{a^{\left(d-2\right)/2}}{2^{d-1}}x^{d-2}+O\left(x^{d-3}\right)$
and $\xi$ is the Haar measure on $S^{p-1}\times S^{q-1}$. Making
the change of variables, $r=\frac{\sqrt{a}e^{t}}{2T}$, gives 
\begin{equation}
\sqrt{a}\cosh t=Tr+a/4Tr\quad\textrm{and}\quad\sqrt{a}\sinh t=Tr-a/4Tr.\label{eq:5.41}
\end{equation}
Let $L_{1}'=\left\langle v_{s+1},\dots,v_{s+r_{1}}\right\rangle $,
$L_{2}'=\left\langle v_{s+r_{1}+1},\dots,v_{d}\right\rangle $, $S^{r_{1}-1}$
be the unit sphere inside $L_{1}'$, $S^{r_{2}-1}$ be the unit sphere
inside $L_{2}'$, $\alpha'\in S^{r_{1}-1}$ and $\beta'\in S^{r_{2}-1}$.
We may write 
\[
d\xi\left(\alpha,\beta\right)=\delta\left(\alpha,\beta\right)d\alpha_{1}\dots d\alpha_{s_{1}}d\beta_{1}\dots d\beta_{s_{2}}d\xi'\left(\alpha',\beta'\right)
\]
where $\delta\left(\alpha,\beta\right)$ is the appropriate density
function and $d\xi'$ is the Haar measure on $S^{r_{1}-1}\times S^{r_{2}-1}.$
This gives
\begin{equation}
dm_{g}\left(v\right)=P\left(\frac{2Tr}{\sqrt{a}}\right)\delta\left(\alpha,\beta\right)\frac{dr}{r}d\alpha_{1}\dots d\alpha_{s_{1}}d\beta_{1}\dots d\beta_{s_{2}}d\xi'\left(\alpha',\beta'\right).\label{eq:5.411}
\end{equation}
Let $z\in\mathbb{R}^{s}$. Make the further change of variables 
\begin{equation}
\left(\alpha_{1},\dots,\alpha_{s_{1}},\beta_{1},\dots,\beta_{s-s_{1}}\right)=\frac{1}{Tr}J^{-1}z,\label{eq:5.42}
\end{equation}
this means that 
\begin{equation}
d\alpha_{1}\dots d\alpha_{s_{1}}d\beta_{1}\dots d\beta_{s_{2}}=\frac{1}{\det\left(J\right)\left(Tr\right)^{s}}dz.\label{eq:5.43}
\end{equation}
Moreover, using \eqref{eq:5.40}, \eqref{eq:5.41} and \eqref{eq:5.42}
gives 
\begin{equation}
M_{0}^{g}\left(v\right)=z+O\left(1/T\right)\quad\textrm{and}\quad v/T=r\left(\alpha'+\beta'\right)+O\left(1/T\right).\label{eq:5.44}
\end{equation}
Since $h$ is continuous and compactly supported it may bounded by
an integrable function and hence
\begin{alignat*}{1}
\lim_{T\rightarrow\infty}\frac{1}{T^{d-s-2}}\Theta\left(h,T\right) & =\lim_{T\rightarrow\infty}\frac{1}{T^{d-s-2}}\int_{X_{g}\left(\mathbb{R}\right)}h\left(M_{0}^{g}\left(v\right),v/T\right)dm_{g}\left(v\right)\\
 & =\int_{X_{g}\left(\mathbb{R}\right)}\lim_{T\rightarrow\infty}\frac{1}{T^{d-s-2}}h\left(M_{0}^{g}\left(v\right),v/T\right)dm_{g}\left(v\right)\\
 & =\int_{S^{r_{1}-1}\times S^{r_{2}-1}}\int_{0}^{\infty}\int_{\mathbb{R}^{s}}h\left(z,r\left(\alpha'+\beta'\right)\right)r^{d-s-2}\delta\left(\alpha',\beta'\right)dz\frac{dr}{2r}d\xi'\left(\alpha',\beta'\right),
\end{alignat*}
where in the last step follows from \eqref{eq:5.411}, the definition
of $P\left(x\right)$, \eqref{eq:5.43} and \eqref{eq:5.44}. Note
that from the definition of $\delta$ it is clear that $\delta\left(\alpha',\beta'\right)=1$.
Finally, let $e_{0}=\frac{1}{\sqrt{2}}\left(f_{1}+f_{p+1}\right)$
and $\frac{1}{\sqrt{2}}\left(\alpha'+\beta'\right)=ke_{0}$ and $r'=\sqrt{2}r$
to get that \textcolor{black}{
\[
\lim_{T\rightarrow\infty}\frac{1}{T^{d-s-2}}\Theta\left(h,T\right)=C_{1}\int_{K_{g}}\int_{0}^{\infty}\int_{\mathbb{R}^{s}}h\left(z,r'ke_{0}\right)r'^{d-s-2}dz\frac{dr'}{2r'}d\nu_{g}\left(k\right).
\]
}\end{proof}
\begin{cor}
\label{cor:volume estimate}For all $g\in\mathcal{C}_{\textrm{SL}}\left(r_{1},r_{2}\right)$
there exists a constant $C_{3}>0$ such that for all compact regions
$R\subset\mathbb{R}^{s}$ with piecewise smooth boundary 
\[
\lim_{T\rightarrow\infty}\dfrac{1}{T^{d-s-2}}\textrm{Vol}_{X_{g}}\left(V_{M_{0}^{g}}\left(R\right)\cap A\left(T/2,T\right)\right)=C_{3}\textrm{Vol}\left(R\right).
\]
\end{cor}
\begin{proof}
Let $\mathbb{1}$ denote the characteristic function of $R\times A\left(1/2,1\right)$,
then it is clear that 
\begin{alignat*}{1}
\lim_{T\rightarrow\infty}\dfrac{1}{T^{d-s-2}}\textrm{Vol}_{X_{g}}\left(V_{M_{0}^{g}}\left(R\right)\cap A\left(T/2,T\right)\right) & =\lim_{T\rightarrow\infty}\dfrac{1}{T^{d-s-2}}\int_{X_{g}\left(\mathbb{R}\right)}\mathbb{1}\left(M_{0}\left(gv\right),v/T\right)dm_{g}\left(v\right)\\
 & =\lim_{T\rightarrow\infty}\frac{1}{T^{d-s-2}}\Theta\left(\mathbb{1},T\right).
\end{alignat*}
Since $R$ has piecewise smooth boundary there exist regions $R_{\delta}^{-}\subseteq R\times A\left(1/2,1\right)\subseteq R_{\delta}^{+}$
such that $\lim_{\delta\rightarrow0}R_{\delta}^{+}=\lim_{\delta\rightarrow0}R_{\delta}^{-}=R$
and for all $\delta>0$ we can choose continuous compactly supported
functions $h_{\delta}^{-}$ and $h_{\delta}^{+}$on $\mathbb{R}^{s}\times\mathbb{R}^{d}\setminus0$,
such that $0\leq h_{\delta}^{-}\leq\mathbb{1}\leq h_{\delta}^{+}\leq1$,
$h_{\delta}^{-}\left(v\right)=\mathbb{1}\left(v\right)$ if $v\in R_{\delta}^{-}$
and $h_{\delta}^{+}\left(v\right)=0$ if $v\notin R_{\delta}^{+}$.
By Lemma \ref{lem:further approximations} 
\begin{alignat*}{1}
\lim_{T\rightarrow\infty}\frac{1}{T^{d-s-2}}\Theta\left(h_{\delta}^{-},T\right) & \leq\liminf_{T\rightarrow\infty}\dfrac{1}{T^{d-s-2}}\int_{X_{g}\left(\mathbb{R}\right)}\mathbb{1}\left(M_{0}\left(gv\right),v/T\right)dm_{g}\left(v\right)\\
 & \leq\limsup_{T\rightarrow\infty}\dfrac{1}{T^{d-s-2}}\int_{X_{g}\left(\mathbb{R}\right)}\mathbb{1}\left(M_{0}\left(gv\right),v/T\right)dm_{g}\left(v\right)\leq\lim_{T\rightarrow\infty}\frac{1}{T^{d-s-2}}\Theta\left(h_{\delta}^{+},T\right).
\end{alignat*}
It is clear that 
\[
\lim_{\delta\rightarrow\infty}\lim_{T\rightarrow\infty}\frac{1}{T^{d-s-2}}\Theta\left(h_{\delta}^{-},T\right)=\lim_{\delta\rightarrow\infty}\lim_{T\rightarrow\infty}\frac{1}{T^{d-s-2}}\Theta\left(h_{\delta}^{+},T\right)=\lim_{T\rightarrow\infty}\frac{1}{T^{d-s-2}}\Theta\left(\mathbb{1},T\right),
\]
hence we can apply Lemma \ref{lem:further approximations} to get
that 
\begin{alignat*}{1}
\lim_{T\rightarrow\infty}\frac{1}{T^{d-s-2}}\Theta\left(\mathbb{1},T\right) & =C_{1}\int_{K_{g}}\int_{0}^{\infty}\int_{\mathbb{R}^{s}}\mathbb{1}\left(z,rk^{-1}e_{0}\right)r^{d-s-2}dz\frac{dr}{2r}d\nu_{g}\left(k\right)\\
 & =C_{1}\int_{\mathbb{R}^{s}}\mathbb{1}_{R}\left(z\right)dz\int_{K_{g}}\int_{0}^{\infty}\mathbb{1}_{A\left(1/2,1\right)}\left(rk^{-1}e_{0}\right)r^{d-s-2}\frac{dr}{2r}d\nu_{g}\left(k\right)=C_{3}\textrm{Vol}\left(R\right).
\end{alignat*}
The last equality holds because 
\[
\int_{K_{g}}\int_{0}^{\infty}\mathbb{1}_{A\left(1/2,1\right)}\left(rk^{-1}e_{0}\right)r^{d-s-2}\frac{dr}{2r}d\nu_{g}\left(k\right)<\infty
\]
as $\mathbb{1}_{A\left(1/2,1\right)}$ has compact support and $K_{g}$
is compact. 
\end{proof}

\subsection{Proof of Theorem \ref{main theorem-1}.}

By Theorem 4.9 of \cite{MR1278263} there exist $v_{1},\ldots,v_{j}\in X_{g}\left(\mathbb{\mathbb{Z}}\right)$
such that $X_{g}\left(\mathbb{\mathbb{Z}}\right)=\bigsqcup_{i=1}^{j}\Gamma_{g}v_{i}$.
Let $P_{i}\left(g\right)=\left\{ x\in G_{g}:xv_{i}=v_{i}\right\} $
and $\Lambda_{i}\left(g\right)=P_{i}\left(g\right)\cap\Gamma_{g}$.
By Proposition 1.13 of \cite{MR1790156} there exist Haar measures
$\varrho_{\Lambda_{i}}$, $p_{\Lambda_{i}}$ and $\gamma_{\Lambda_{i}}$on
$G_{g}/\Lambda_{i}\left(g\right)$, $P_{i}\left(g\right)/\Lambda_{i}\left(g\right)$
and $\Gamma_{g}/\Lambda_{i}\left(g\right)$ respectively such that,
for \foreignlanguage{english}{$f\in C_{c}\left(G_{g}/\Lambda_{i}\left(g\right)\right)$},
and hence for integrable functions on $G_{g}/\Lambda_{i}\left(g\right)$,
\begin{equation}
\int_{G_{g}/\Lambda_{i}\left(g\right)}fd\varrho_{\Lambda_{i}}=\int_{X_{g}\left(\mathbb{R}\right)}\int_{P_{i}\left(g\right)/\Lambda_{i}\left(g\right)}f\left(xp\right)dp_{\Lambda_{i}}\left(p\right)dm_{g}\left(x\right),\label{eq:s1}
\end{equation}
and
\begin{equation}
\int_{G_{g}/\Lambda_{i}\left(g\right)}fd\varrho_{\Lambda_{i}}=\int_{G_{g}/\Gamma_{g}}\int_{\Gamma_{g}/\Lambda_{i}\left(g\right)}f\left(x\gamma\right)d\gamma_{\Lambda_{i}}\left(\gamma\right)d\mu_{g}\left(x\right).\label{eq:s2}
\end{equation}
Note that $\Gamma_{g}/\Lambda_{i}\left(g\right)=\Gamma_{g}v_{i}$
is discrete and its Haar measure $d\gamma_{\Lambda_{i}}$ is just
the counting measure and so 
\begin{equation}
\int_{\Gamma_{g}/\Lambda_{i}\left(g\right)}f\left(x\gamma\right)d\gamma_{\Lambda_{i}}\left(\gamma\right)=\sum_{v\in\Gamma_{g}v_{i}}f\left(xv\right).\label{eq:siegle 2-1}
\end{equation}
Therefore the normalisations already present on $m_{g}$ and $\mu_{g}$
induce a normalisation on $p_{\Lambda_{i}}$. Moreover, it follows
from the Borel Harish-Chandra Theorem (cf. \cite{MR1278263}, Theorem
4.13) that the measure of $p_{\Lambda_{i}}\left(P_{i}\left(g\right)/\Lambda_{i}\left(g\right)\right)<\infty$,
for each $1\leq i\leq j$. As in \cite{MR1609447} and \cite{MR1237827}
where the proofs rely on Siegel's integral formula, here the proof
relies on the following result.
\begin{lem}
\label{siegel mod} For all $f\in C_{c}\left(X_{g}\left(\mathbb{R}\right)\right)$
and $g\in\mathcal{C}_{\textrm{SL}}\left(r_{1},r_{2}\right)$ there
exists a constant 
\[
C\left(g\right)=\sum_{i=1}^{j}p_{\Lambda_{i}}\left(P_{i}\left(g\right)/\Lambda_{i}\left(g\right)\right),
\]
such that 
\begin{equation}
C\left(g\right)\int_{X_{g}\left(\mathbb{R}\right)}fdm_{g}=\int_{G_{g}/\Gamma_{g}}F_{f,g}d\mu_{g}.\label{eq:siegl formula}
\end{equation}
\end{lem}
\begin{proof}
Note that for $1\leq i\leq j$, $G_{g}/P_{i}\left(g\right)\cong X_{g}\left(\mathbb{R}\right)$.
If $f\in C_{c}\left(X_{g}\left(\mathbb{R}\right)\right)$ then $f$
is $\Lambda_{i}\left(g\right)$ invariant and therefore can be considered
as an integrable function on $G_{g}/\Lambda_{i}\left(g\right)$ and
so 
\begin{equation}
\int_{X_{g}\left(\mathbb{R}\right)}\int_{P_{i}\left(g\right)/\Lambda_{i}\left(g\right)}f\left(xp\right)dp_{\Lambda_{i}}\left(p\right)dm_{g}\left(x\right)=\int_{P_{i}\left(g\right)/\Lambda_{i}\left(g\right)}dp_{\Lambda_{i}}\int_{X_{g}\left(\mathbb{R}\right)}fdm_{g}.\label{eq:s4}
\end{equation}
Now it follows from the definition of $F_{f,g}$ (i.e. \eqref{eq:def of F}),
\prettyref{eq:s1}, \prettyref{eq:s2}, \prettyref{eq:siegle 2-1}
and \prettyref{eq:s4} that 
\begin{alignat*}{1}
\int_{G_{g}/\Gamma_{g}}F_{f,g}d\mu_{g} & =\sum_{i=1}^{j}\int_{G_{g}/\Gamma_{g}}\sum_{v\in\Gamma_{g}v_{i}}f\left(xv\right)d\mu_{g}\left(x\right)\\
 & =\sum_{i=1}^{j}\int_{P_{i}\left(g\right)/\Lambda_{i}\left(g\right)}dp_{\Lambda_{i}}\int_{X_{g}\left(\mathbb{R}\right)}fdm_{g},
\end{alignat*}
which is the desired result.
\end{proof}
The final Lemma of this section is the counterpart of Lemma 3.9 from
\cite{MR1609447} and again the proof there is mimicked. 
\begin{lem}
\label{lem:Link}Let $f$ be a continuous function of compact support
on $\mathbb{R}_{+}^{d}$. Then for all $g\in\mathcal{C}_{\textrm{SL}}\left(r_{1},r_{2}\right)$,
\[
\lim_{T\rightarrow\infty}\frac{1}{T^{d-s-2}}\int_{X_{g}\left(\mathbb{R}\right)}J_{f,g}\left(M_{0}^{g}\left(v\right),\left\Vert v\right\Vert /T\right)dm_{g}\left(v\right)=C_{1}C\left(g\right)\int_{G_{g}/\Gamma_{g}}F_{f,g}d\mu_{g},
\]
where $C_{1}$ is defined by \eqref{eq:def of C_1} and $C\left(g\right)$
is defined in Lemma \ref{siegel mod}.\end{lem}
\begin{proof}
Let $v_{i}$ be the components of $v$ when written in the basis $b_{1},\ldots,b_{d}$
from Lemma \ref{lemma about evctors}. Using the change of variables
$\left(v_{1},\ldots,v_{d}\right)\rightarrow\left(z_{1},\ldots,z_{s},r,v_{s+2},\ldots,a\right)$
where $Q_{0}^{g}\left(v_{1},\ldots,v_{d}\right)=a$ we see that 
\[
\int_{\mathbb{R}^{d}}f\left(v\right)dv=\int_{-\infty}^{\infty}\int_{0}^{\infty}\int_{\mathbb{R}^{s}}J_{f,g}\left(z,r\right)r^{d-s-2}dz\frac{dr}{2r}da.
\]
Hence it follows from how $m_{g}$ is defined (i.e. (\ref{eq:def of m_0}))
that 
\begin{equation}
\int_{X_{g}\left(\mathbb{R}\right)}f\left(v\right)dm_{g}\left(v\right)=\int_{0}^{\infty}\int_{\mathbb{R}^{s}}J_{f,g}\left(z,r\right)r^{d-s-2}dz\frac{dr}{2r}.\label{eq:midlle 3.6}
\end{equation}
Lemma \ref{lem:further approximations} and (\ref{eq:midlle 3.6})
imply that 
\[
\lim_{T\rightarrow\infty}\frac{1}{T^{d-s-2}}\int_{X_{g}\left(\mathbb{R}\right)}J_{f,g}\left(M_{0}^{g}\left(v\right),\left\Vert v\right\Vert /T\right)dm_{g}\left(v\right)=C_{1}\int_{K_{g}}\left(\int_{X_{g}\left(\mathbb{R}\right)}f\left(v\right)dm_{g}\right)d\nu_{g}\left(k\right).
\]
Now the conclusion follows from Lemma \ref{siegel mod}. 
\end{proof}
The purpose of Lemma \ref{lem:Link} is to relate the integral over
$G_{g}/\Gamma_{g}$ to an integral over $X_{g}\left(\mathbb{R}\right)$
in order that the integral over $X_{g}\left(\mathbb{R}\right)$ can
be approximated by an integral over $K_{g}$ via Theorem \ref{thm:ergodic result-1}.
Then the integral over $K_{g}$ can be approximated by the appropriate
counting function via Corollary \ref{lem:sum approx}. We now proceed
to put this into action in the proof of our main Theorem which is
just a modification of the proof in \cite{MR1609447}. 
\begin{proof}[Proof of Theorem \ref{main theorem-1}.]
 By Lemma \ref{lem:further approximations} the functional $\Psi$
on $C_{c}\left(\mathbb{R}^{s}\times\mathbb{R}^{d}\setminus\left\{ 0\right\} \right)$
given by 
\[
\Psi\left(h\right)=\lim_{T\rightarrow\infty}\frac{1}{T^{d-s-2}}\Theta\left(h,T\right)
\]
is continuous. For all connected regions $R\subset\mathbb{R}^{s}$
with smooth boundary, if $\mathbb{1}$ denotes the characteristic
function of $R\times A\left(1/2,1\right)$, then for every $\epsilon>0$
there exist continuous functions $h_{+}$ and $h_{-}$ on $\mathbb{R}^{s}\times\mathbb{R}^{d}\setminus\left\{ 0\right\} $
such that for all $\left(r,v\right)\in\mathbb{R}^{s}\times\mathbb{R}^{d}\setminus\left\{ 0\right\} $,
\begin{equation}
h_{-}\left(r,v\right)\leq\mathbb{1}\left(r,v\right)\leq h_{+}\left(r,v\right)\label{eq:final proof 1}
\end{equation}
and
\begin{equation}
\left|\Psi\left(h_{+}\right)-\Psi\left(h_{-}\right)\right|<\epsilon.\label{eq:final proof 2}
\end{equation}
Let $\mathcal{J}$ denote the space of linear combinations of functions
on $\mathbb{R}^{s}\times\mathbb{R}^{d}$ of the form $J_{f,g}\left(r,\left\Vert v\right\Vert \right)$,
where $f$ is continuous function of compact support on $\mathbb{R}_{+}^{d}$.
Let $\mathcal{H}$ denote the collection of functions in $C_{c}\left(\mathbb{R}^{s}\times\mathbb{R}^{d}\setminus\left\{ 0\right\} \right)$
such that if $h\in\mathcal{H}$ then $h$ takes an argument of the
form $\left(r,\left\Vert v\right\Vert \right)$. By Lemma \ref{lem:stone weirstarss}
$\mathcal{J}$ is dense in $\mathcal{H}$ and since $h_{+}$ and $h_{-}$
belong to $\mathcal{H}$ we may suppose that $h_{+}$ and $h_{-}$
maybe written as a finite linear combination of functions from $\mathcal{J}$.
The function $F_{f,g}$ defined by \eqref{eq:def of F} obeys the
bound \eqref{eq:bound in thm 2.4} with $\delta=1$, by \eqref{eq:upper bound for function}.
Moreover, Lemma 3.10 of \cite{MR1609447} implies that \textcolor{black}{$F_{f,g}\in L_{1}\left(G_{g}/\Gamma_{g}\right)$.
Therefore, if $h'\in\left\{ h_{+},h_{-}\right\} $, then for all $g\in\mathcal{C}_{\textrm{SL}}\left(r_{1},r_{2}\right)$
we can apply Theorem \ref{thm:ergodic result-1} with the function
$F_{f,g}$, followed by Corollary \ref{lem:sum approx} and Lemma
\ref{lem:Link} to get that there exists $t_{0}>0$, so that for all
$\epsilon>0$ and $t>t_{0}$,} 
\begin{equation}
\left|\frac{C\left(g\right)}{e^{\left(d-s-2\right)t}}\sum_{v\in X_{g}\left(\mathbb{Z}\right)}h'\left(M_{0}^{g}\left(v\right),ve^{-t}\right)-\Psi\left(h'\right)\right|<\epsilon.\label{eq:final proof 3}
\end{equation}
From the definition of $\Psi\left(h\right)$ we see that for \textcolor{black}{all
$h\in C_{c}\left(\mathbb{R}^{s}\times\mathbb{R}^{d}\setminus\left\{ 0\right\} \right)$
and} $g\in\mathcal{C}_{\textrm{SL}}\left(r_{1},r_{2}\right)$ there
exists $t_{0}>0$, so that for all $\epsilon>0$ and $t>t_{0}$, 
\begin{equation}
\left|\frac{1}{e^{\left(d-s-2\right)t}}\int_{X_{g}\left(\mathbb{R}\right)}h\left(M_{0}^{g}\left(v\right),ve^{-t}\right)dm_{g}\left(v\right)-\Psi\left(h\right)\right|<\epsilon.\label{eq:final proof 4}
\end{equation}
Clearly (\ref{eq:final proof 1}) implies 
\begin{alignat}{1}
\frac{C\left(g\right)}{e^{\left(d-s-2\right)t}}\sum_{v\in X_{g}\left(\mathbb{Z}\right)}h_{-}\left(M_{0}^{g}\left(v\right),ve^{-t}\right)-\Psi\left(h_{+}\right) & \leq\frac{C\left(g\right)}{e^{\left(d-s-2\right)t}}\sum_{v\in X_{g}\left(\mathbb{Z}\right)}\mathbb{1}\left(M_{0}^{g}\left(v\right),ve^{-t}\right)-\Psi\left(h_{+}\right)\nonumber \\
 & \leq\frac{C\left(g\right)}{e^{\left(d-s-2\right)t}}\sum_{v\in X_{g}\left(\mathbb{Z}\right)}h_{+}\left(M_{0}^{g}\left(v\right),ve^{-t}\right)-\Psi\left(h_{+}\right).\label{eq:final proof 5}
\end{alignat}
Apply (\ref{eq:final proof 2}) to the left hand side of (\ref{eq:final proof 5})
and then apply and (\ref{eq:final proof 3}) with suitable choices
of $\epsilon$'s to get that for all $g\in\mathcal{C}_{\textrm{SL}}\left(r_{1},r_{2}\right)$
there exists $t_{0}>0$, so that for all $\theta>0$ and $t>t_{0}$,
\begin{equation}
\left|\frac{C\left(g\right)}{e^{\left(d-s-2\right)t}}\sum_{v\in X_{g}\left(\mathbb{Z}\right)}\mathbb{1}\left(M_{0}^{g}\left(v\right),ve^{-t}\right)-\Psi\left(h_{+}\right)\right|\leq\frac{\theta}{2}.\label{final proof 6}
\end{equation}
Similarly using (\ref{eq:final proof 1}), (\ref{eq:final proof 2})
and (\ref{eq:final proof 4}) we see that for all $g\in\mathcal{C}_{\textrm{SL}}\left(r_{1},r_{2}\right)$
there exists $t_{0}>0$, so that for all $\theta>0$ and $t>t_{0}$,
\begin{equation}
\left|\frac{1}{e^{\left(d-s-2\right)t}}\int_{X_{g}\left(\mathbb{R}\right)}\mathbb{1}\left(M_{0}^{g}\left(v\right),ve^{-t}\right)dm_{g}\left(v\right)-\Psi\left(h_{+}\right)\right|\leq\frac{\theta}{2}.\label{final proof 6-1}
\end{equation}
Hence using (\ref{final proof 6}) and (\ref{final proof 6-1}) we
see that for all $g\in\mathcal{C}_{\textrm{SL}}\left(r_{1},r_{2}\right)$
there exists $t_{0}>0$, so that for all $\theta>0$ and $t>t_{0}$
\begin{align}
\left|C\left(g\right)\sum_{v\in X_{g}\left(\mathbb{Z}\right)}\mathbb{1}\left(M_{0}^{g}\left(v\right),ve^{-t}\right)-\int_{X_{g}}\mathbb{1}\left(M_{0}^{g}\left(v\right),ve^{-t}\right)dm_{g}\left(v\right)\right|\leq\theta.\label{eq:final proof 7}
\end{align}
This means that for all $g\in\mathcal{C}_{\textrm{SL}}\left(r_{1},r_{2}\right)$
there exists $t_{0}>0$, so that for all $\theta>0$ and $t>t_{0}$,
\begin{multline}
\left(1-\theta\right)\int_{X_{g}\left(\mathbb{R}\right)}\mathbb{1}\left(M_{0}^{g}\left(v\right),ve^{-t}\right)dm_{g}\left(v\right)\\
\leq C\left(g\right)\sum_{v\in X_{g}\left(\mathbb{Z}\right)}\mathbb{1}\left(M_{0}^{g}\left(v\right),ve^{-t}\right)\leq\left(1+\theta\right)\int_{X_{g}\left(\mathbb{R}\right)}\mathbb{1}\left(M_{0}^{g}\left(v\right),ve^{-t}\right)dm_{g}\left(v\right).\label{eq:final proof 8}
\end{multline}
Hence for all $\left(Q,M\right)\in\mathcal{\mathcal{C}_{\textrm{Pairs}}}\left(r_{1},r_{2}\right)$
there exists $t_{0}>0$, so that for all $\theta>0$ and $t>t_{0}$,
\begin{alignat*}{1}
\left(1-\theta\right)\textrm{Vol}_{X_{Q}}\left(V_{M}\left(R\right)\cap A\left(T/2,T\right)\right) & \leq C\left(g\right)\left|X_{Q}\left(\mathbb{Z}\right)\cap V_{M}\left(R\right)\cap A\left(T/2,T\right)\right|\\
 & \leq\left(1+\theta\right)\textrm{Vol}_{X_{Q}}\left(V_{M}\left(R\right)\cap A\left(T/2,T\right)\right).
\end{alignat*}
The conclusion of the Theorem follows by applying Corollary \ref{cor:volume estimate}
and summing a geometric series.
\end{proof}

\section{Counterexamples\label{sec:Counterexamples}}

In small dimensions there are slightly more integer points than expected
on the quadratic surfaces defined by forms with signature $\left(1,2\right)$
and $\left(2,2\right)$. This fact was exploited in \cite{MR1609447}
to show that the expected asymptotic formula for the situation they
consider is not valid for these special cases. In a similar manner
it is possible to construct examples that show that Theorem \ref{main theorem}
is not valid in the cases that the signature of $H_{g}$ is $\left(1,2\right)$
or $\left(2,2\right)$. In this section, for the sake of brevity we
restrict our attention to the case when $s=1$, but we note that similar
arguments would hold in the case when $s>1$. To start with make the
following definitions 
\begin{eqnarray*}
Q_{1}\left(x\right) & = & -x_{1}x_{2}+x_{3}^{2}+x_{4}^{2},\\
Q_{2}\left(x\right) & = & x_{1}x_{2}+x_{3}^{2}-x_{4}^{2},\\
Q_{3}\left(x\right) & = & -x_{1}x_{2}+x_{3}^{2}+x_{4}^{2}-\alpha x_{5}^{2},\\
L_{\alpha}\left(x\right) & = & x_{1}-\alpha x_{2}.
\end{eqnarray*}
We can now prove. 
\begin{lem}
\label{lem:ce1}Let $\epsilon>0$, suppose $\left[a,b\right]=\left[1/2-\epsilon,1\right]$
or $\left[-1,-1/2+\epsilon\right]$. Let $a>0$, then for every $T_{0}>0$,
the set of $\beta\in\mathbb{R}$ for which there exists a $T>T_{0}$
such that 
\[
\left|X_{Q_{1}}^{a}\left(\mathbb{Z}\right)\cap V_{L_{\beta}}\left(\left[a,b\right]\right)\cap A\left(0,T\right)\right|>T\left(\log T\right)^{1-\epsilon}\quad\textrm{or}\quad\left|X_{Q_{2}}^{a}\left(\mathbb{Z}\right)\cap V_{L_{\beta}}\left(\left[a,b\right]\right)\cap A\left(0,T\right)\right|>T\left(\log T\right)^{1-\epsilon}
\]
is dense. Similarly if $a=0$, then for every $T_{0}>0$, the set
of $\beta\in\mathbb{R}$ for which there exists a $T>T_{0}$ such
that 
\[
\left|X_{Q_{3}}^{a}\left(\mathbb{Z}\right)\cap V_{L_{\beta}}\left(\left[a,b\right]\right)\cap A\left(0,T\right)\right|>T^{2}\left(\log T\right)^{1-\epsilon}
\]
is dense. \end{lem}
\begin{proof}
Let $S_{i}\left(\alpha,T,a\right)=\left\{ x\in\mathbb{Z}^{d_{i}}:L_{\alpha}\left(x\right)=0,Q_{i}\left(x\right)=a,\left\Vert x\right\Vert \leq T\right\} $
where $d_{i}=4$ if $i=1$ or $2$ and $d_{i}=5$ if $i=3$. Lemma
3.14 of \cite{MR1609447} implies that
\begin{eqnarray}
\left|S_{i}\left(\alpha,T,a\right)\right| & \sim T\log T & \quad\textrm{for }i=1,2\textrm{ and }\sqrt{\alpha}\in\mathbb{Q}\textrm{ and }a>0,\label{eq:ce1}\\
\left|S_{3}\left(\alpha,T,0\right)\right| & \sim T^{2}\log T & \quad\textrm{for }\sqrt{\alpha}\in\mathbb{Q}.\label{eq:ce1.1}
\end{eqnarray}
Note that if $i=1,2$ and $x\in S_{i}\left(\alpha,T,a\right)\setminus S_{i}\left(\alpha,T/2,a\right)$,
then 
\begin{equation}
\frac{T^{2}}{4}-\left(\alpha^{2}+1\right)x_{2}^{2}\leq x_{3}^{2}+x_{4}^{2}\leq T^{2}-\left(\alpha^{2}+1\right)x_{2}^{2}\label{eq:ce2}
\end{equation}
and 
\begin{equation}
x_{3}^{2}+x_{4}^{2}=\alpha x_{2}^{2}+a.\label{eq:ce3}
\end{equation}
Similarly if $x\in S_{3}\left(\alpha,T,0\right)\setminus S_{3}\left(\alpha,T/2,0\right)$,
\begin{equation}
\frac{T^{2}}{4}-\left(\alpha^{2}+1\right)x_{2}^{2}\leq x_{3}^{2}+x_{4}^{2}+x_{5}^{2}\leq T^{2}-\left(\alpha^{2}+1\right)x_{2}^{2}\label{eq:ce2-1}
\end{equation}
and 
\begin{equation}
x_{3}^{2}+x_{4}^{2}=\alpha\left(x_{2}^{2}+x_{5}^{2}\right).\label{eq:ce3-1}
\end{equation}
Combining (\ref{eq:ce2}) and (\ref{eq:ce3}) gives 
\begin{equation}
\frac{T^{2}-4a}{4\left(\alpha^{2}+\alpha+1\right)}\leq x_{2}^{2}\leq\frac{T^{2}-a}{\alpha^{2}+\alpha+1}.\label{eq:ce4}
\end{equation}
Respectively, combining (\ref{eq:ce2-1}) and (\ref{eq:ce3-1}) gives
\begin{equation}
\frac{T^{2}-\left(\alpha+1\right)x_{5}^{2}}{4\left(\alpha^{2}+\alpha+1\right)}\leq x_{2}^{2}\leq\frac{T^{2}-\left(\alpha+1\right)x_{5}^{2}}{\alpha^{2}+\alpha+1},\label{eq:ce4-1}
\end{equation}
which upon noting that $-T\leq x_{5}\leq T$ offers
\begin{equation}
\frac{T^{2}-\left(\alpha+1\right)T}{4\left(\alpha^{2}+\alpha+1\right)}\leq x_{2}^{2}\leq\frac{T^{2}+\left(\alpha+1\right)T}{\alpha^{2}+\alpha+1}.\label{eq:ce4-2}
\end{equation}
Take 
\begin{equation}
\beta_{\pm}=\alpha\pm\sqrt{\frac{\alpha^{2}+\alpha+1}{T^{2}}}.\label{eq:cebeat}
\end{equation}
It is clear that $L_{\beta_{\pm}}\left(x\right)=L_{\alpha}\left(x\right)\pm\sqrt{\frac{\alpha^{2}+\alpha+1}{T^{2}}}x_{2}$
and hence if $i=1,2$ and $x\in S_{i}\left(\alpha,T,a\right)\setminus S_{i}\left(\alpha,T/2,a\right)$,
then (\ref{eq:ce4}) implies 
\begin{equation}
\sqrt{\frac{1}{4}-\frac{a}{T^{2}}}\leq L_{\beta_{+}}\left(x\right)\leq\sqrt{1-\frac{a}{T^{2}}}\quad\textrm{and}\quad-\sqrt{1-\frac{a}{T^{2}}}\leq L_{\beta_{-}}\left(x\right)\leq-\sqrt{\frac{1}{4}-\frac{a}{T^{2}}}.\label{eq:ce5}
\end{equation}
Similarly if $x\in S_{3}\left(\alpha,T,0\right)\setminus S_{3}\left(\alpha,T/2,0\right)$,
then (\ref{eq:ce4-2}) implies 
\begin{equation}
\sqrt{\frac{1}{4}-\frac{\left(\alpha+1\right)}{T}}\leq L_{\beta_{+}}\left(x\right)\leq\sqrt{1-\frac{\left(\alpha+1\right)}{T}}\quad\textrm{and}\quad-\sqrt{1-\frac{\left(\alpha+1\right)}{T}}\leq L_{\beta_{-}}\left(x\right)\leq-\sqrt{\frac{1}{4}-\frac{\left(\alpha+1\right)}{T}}.\label{eq:ce5-1}
\end{equation}
This means for all $\epsilon>0$ there exists $T_{+}>0$ such that
if $T>T_{+}$ then $S_{i}\left(\alpha,T,a\right)\subset X_{Q_{i}}^{a}\left(\mathbb{Z}\right)\cap V_{L_{\beta_{+}}}\left(\left[1/2-\epsilon,1\right]\right)\cap A\left(0,T\right)$
respectively there also exists $T_{-}>0$ such that if $T>T_{-}$
then $S_{i}\left(\alpha,T,a\right)\subset X_{Q_{i}}^{a}\left(\mathbb{Z}\right)\cap V_{L_{\beta_{-}}}\left(\left[-1,-1/2+\epsilon\right]\right)\cap A\left(0,T\right)$.
By (\ref{eq:ce1}) and (\ref{eq:ce1.1}) for $i=1,2$ and large enough
$T$, $\left|S_{i}\left(\alpha,T,a\right)\right|>T\left(\log T\right)^{1-\epsilon}$
and $\left|S_{i}\left(\alpha,T,a\right)\right|>CT^{2}\left(\log T\right)^{1-\epsilon}$.
The set of $\beta$ satisfying (\ref{eq:cebeat}) for rational $\alpha$
and large $T$ is clearly dense and this proves the Lemma. \end{proof}
\begin{thm}
Let $j=1,2$. For every $\epsilon>0$ and every interval $\left[a,b\right]$
there exists a rational quadratic form $Q$ and an irrational linear
form $L$ such that $\textrm{Stab}_{SO\left(Q\right)}\left(L\right)\cong SO\left(j,2\right)$
such that for an infinite sequence $T_{k}\rightarrow\infty$, 
\[
\left|X_{Q}^{a_{j}}\left(\mathbb{Z}\right)\cap V_{L}\left(\left[a,b\right]\right)\cap A\left(0,T_{k}\right)\right|>T_{k}^{j}\left(\log T_{k}\right)^{1-\epsilon},
\]
where $a_{1}>0$ and $a_{2}=0$.\end{thm}
\begin{proof}
Since the interval $\left[a,b\right]$ must intersect either the positive
or negative reals there is no loss of generality in assuming, after
passing to a subset and rescaling that $\left[a,b\right]=\left[1/4,5/4\right]$
or $\left[-5/4,-1/4\right]$. For a given $S>0$ and $i=1,2$ let
$\mathcal{U}_{S}$ be the set of $\gamma\in\mathbb{R}$ for which
there exist $\beta\in\mathbb{R}$ and $T>S$ with 
\begin{equation}
\left|X_{Q_{i}}^{a_{1}}\left(\mathbb{Z}\right)\cap V_{L_{\beta}}\left(\left[1/2,1\right]\right)\cap A\left(0,T\right)\right|>CT\log T,\label{eq:ce2.1}
\end{equation}
and
\begin{equation}
\left|\beta-\gamma\right|<T^{-2}.\label{eq:ce2.2}
\end{equation}
Then $\mathcal{U}_{S}$ is open and dense by Lemma \ref{lem:ce1}.
By the Baire category Theorem (cf. \cite{MR924157}, Theorem 5.6)
$\bigcap_{k=1}^{\infty}\mathcal{U}_{2^{k+1}}$ is dense in $\mathbb{R}$
and is in fact of second category and hence uncountable. Let $\gamma\in\bigcap_{k=1}^{\infty}\mathcal{U}_{2^{k+1}}\setminus\mathbb{Q}$,
then there exist infinite sequences $\beta_{k}$ and $T_{k}$ such
that (\ref{eq:ce2.1}) and (\ref{eq:ce2.2}) hold with $\beta$ replaced
by $\beta_{k}$ and $T$ by $T_{k}$. Note that (\ref{eq:ce2.2})
implies that for $\left\Vert x\right\Vert <T_{k}$, 
\[
\left|L_{\beta_{k}}\left(x\right)-L_{\gamma}\left(x\right)\right|<\frac{1}{T_{k}}<\frac{1}{4},
\]
so that 
\[
X_{Q_{i}}^{a_{1}}\left(\mathbb{Z}\right)\cap V_{L_{\beta_{k}}}\left(\left[1/2,1\right]\right)\cap A\left(0,T_{k}\right)\subseteq X_{Q_{i}}^{a_{1}}\left(\mathbb{Z}\right)\cap V_{L_{\gamma}}\left(\left[1/4,5/4\right]\right)\cap A\left(0,T_{k}\right)
\]
and hence $\left|X_{Q_{i}}^{a_{1}}\left(\mathbb{Z}\right)\cap V_{L_{\gamma}}\left(\left[1/4,5/4\right]\right)\cap A\left(0,T_{k}\right)\right|>CT_{k}\log T_{k}$
by (\ref{eq:ce2.1}). If $i=3$ then we can carry out the same process
but we replace $\mathcal{U}_{S}$ by the set $\mathcal{W}_{S}$, of
$\gamma\in\mathbb{R}$ for which there exist $\beta\in\mathbb{R}$
and $T>S$ with 
\[
\left|X_{Q_{3}}^{0}\left(\mathbb{Z}\right)\cap V_{L_{\beta}}\left(\left[1/2,1\right]\right)\cap A\left(0,T\right)\right|>CT^{2}\log T,
\]
and
\[
\left|\beta-\gamma\right|<T^{-2}.
\]

\end{proof}
\bibliographystyle{amsalpha}
\bibliography{References}

\end{document}